\documentclass[10pt, reqno]{amsart}
\usepackage{amsmath}
\usepackage{amsthm}
\usepackage{bbm}
\usepackage{nicefrac}
\usepackage{graphicx}
\usepackage{amssymb}
\usepackage{color}
\usepackage{verbatim}
\usepackage{cite}
\usepackage[font=small,skip=5pt]{caption}
\usepackage{enumitem}
\usepackage[makeroom]{cancel}

\usepackage{graphicx,calc}
\newlength\myheight
\newlength\mydepth
\settototalheight\myheight{Xygp}
\newtheorem{theorem}{Theorem}[section]
\newtheorem{lemma}[theorem]{Lemma}
\newtheorem{proposition}[theorem]{Proposition}
\newtheorem{corollary}[theorem]{Corollary}

\newtheorem{thmx}{Theorem}

\theoremstyle{definition}
\newtheorem{definition}[theorem]{Definition}
\newtheorem{example}[theorem]{Example}
\usepackage{tikz}
\usepackage{stackrel}

\theoremstyle{remark}
\newtheorem{remark}[theorem]{Remark}
\newtheorem{notation}[theorem]{Notation}

\title{Balanced complexes and effective divisors on $\overline{M}_{0,n}$}
\author{Jos\'e Luis Gonz\'alez, Elijah Gunther and Olivia Zhang}
\address{Department of Mathematics, University of California, Riverside, Riverside, CA 92521}
\email{jose.gonzalez@ucr.edu}
\address{Department of Mathematics, Yale University, New Haven, CT 06511}
\email{elijah.gunther@yale.edu}
\address{Department of Mathematics, Yale University, New Haven, CT 06511}
\email{olivia.zhang@yale.edu}

\let\oldbibliography\thebibliography
\renewcommand{\thebibliography}[1]{\oldbibliography{#1}
\setlength{\itemsep}{3pt}} 

\begin{document}

\begin{abstract}
Doran, Jensen and Giansiracusa showed a bijection between homogeneous elements in the Cox ring of $\overline{M}_{0,n}$ not divisible by any exceptional divisor section, and weighted pure-dimensional simplicial complexes satisfying a zero-tension condition.
Motivated by the study of the monoid of effective divisors, the pseudoeffective cone and the Cox ring of $\overline{M}_{0,n}$,
we point out a simplification of the zero-tension condition and study the space of balanced complexes.
We give examples of irreducible elements in the monoid of effective divisors of $\overline{M}_{0,n}$ for large $n$.
In the case of $\overline{M}_{0,7}$, we classify all such irreducible elements arising from nonsingular complexes and give an example of how irreducibility can be shown in the singular case.
\end{abstract}

\maketitle


 \section{Introduction}            \label{section.introduction}

In this article we study the monoid $M(\overline{M}_{0,n})$ of effective divisor classes in the moduli space $\overline{M}_{0,n}$ via its connection to balanced simplicial complexes introduced in \cite{DGJ}. 
Recall, the moduli spaces $\overline{M}_{0,n}$ parametrize stable rational curves, that is, nodal trees of $\mathbb{P}^1$'s with $n$ marked points and without automorphisms. 
By definition, for a smooth projective variety $X$ the monoid $M(X)$ consists of the classes $D$ in $\operatorname{Pic}(X)$ that admit an effective divisor representative, or in other words, 
that can be represented by a combination of codimension one subvarieties with nonnegative integer coefficients.
For $X$ projective, the only unit in the monoid $M(X)$ is zero, and the irreducible elements in $M(X)$ are those that cannot be written as a sum of two elements unless one of them is zero.
When $\operatorname{Pic}(X)$ is finitely generated and torsion-free there is a unique minimal generating
set for $M(X)$, which is given by the possibly infinite collection of irreducible elements of $M(X)$.
The monoid $M(\overline{M}_{0,n})$ is finitely generated for $n \leq 6 $ and the minimal generating set is known in each of these cases \cite{Castravet}.
For moduli spaces of stable $n$-pointed curves of higher genus, it is known that $M(\overline{\mathcal{M}}_{1,n})$ is not finitely generated for $n \geq 3$ \cite{ChenCoskun} and $M(\overline{\mathcal{M}}_{g,n})$ is not finitely generated if $g \geq 2 $ and $n \geq g+1$ \cite{Mullane}.
Our motivation comes from the following question. 

\vspace{-2mm}

\subsection*{Question 1.} Is $M(\overline{M}_{0,n})$ finitely generated for $n\geq 7$? In particular, is $M(\overline{M}_{0,7})$ finitely generated?

\vspace{1mm}

Given a smooth projective variety $X$ with a finitely generated Picard group $\operatorname{Pic}(X)$
an important invariant in birational geometry is its Cox ring or total coordinate ring 
\[
\operatorname{Cox}(X)=\bigoplus_{L\in Pic(X)}H^{0}(X,L),
\]
which algebraically encodes the geometry of $X$.
For instance, the algebra $\operatorname{Cox}(X)$ carries the data of all morphisms with connected fibers from $X$ to other projective varieties.
The first basic question when studying the Cox ring of a variety is to determine its finite generation.
For an example, the Cox rings of toric varieties and of log-Fano varieties are finitely generated.
When $\operatorname{Cox}(X)$ is finitely generated, all multigraded section rings on $X$ are finitely generated and the birational geometry of $X$ becomes an instance of the theory of variation of geometric invariant quotients, as explained in \cite{HuKeel}.
The monoid $M(X)$ is precisely the supporting monoid of $\operatorname{Cox}(X)$
and the collection of multidegrees in $\operatorname{Pic}(X)$ of a set of generators of $\operatorname{Cox}(X)$ yields a set of generators of $M(X)$.
The Cox ring of $\overline{M}_{0,n}$ is finitely generated for $n \leq 6$ by \cite{Castravet} and not finitely generated for $n \geq 10$ as proved in \cite{Mon.Non.MDS,GK14,HKL}.
It seems that settling the open cases $n=7,8,9$ will require new ideas as suggested in \cite[Remark 6.5]{HKL}.
Deciding the finite generation of $M(\overline{M}_{0,7})$ is a very natural problem from this point of view.
A fundamental problem in the birational geometry of moduli spaces is to understand the structure of the pseudoeffective cone, and more generally,  
the structure of the closed convex cones of effective codimension $k$ cycles, in the space of numerical classes of codimension $k$ cycles $\overline{\operatorname{NE}}^{k}_{\mathbb{R}}(X)$. 
The pseudoeffective cone of $\overline{M}_{0,n}$ is the closed convex hull of $M(\overline{M}_{0,n})$ in $\operatorname{Pic}(\overline{M}_{0,n}) \otimes {\mathbb{R}}=\overline{\operatorname{NE}}^{1}_{\mathbb{R}}(\overline{M}_{0,n})$, and we see that Question 1 also arises naturally from this point of view.

To study Question 1, we use the simplicial approach to effective divisors on $\overline{M}_{0,n}$ introduced by Doran, Giansiracusa and Jensen in \cite{DGJ}. 
By work of Kapranov \cite{Kap} over $\operatorname{Spec}\mathbb{C}$ and Hassett \cite{Has03} over $\operatorname{Spec}\mathbb{Z}$, the moduli space $\overline{M}_{0,n}$ can be constructed by fixing $n-1$ points in $\mathbb{P}^{n-3}$ in linear general position and successively blowing up all linear subspaces spanned by subsets of these points, or more precisely blowing up their strict transforms, in a suitable order. 
Doran and Giansiracusa showed in \cite{DG13} that $\operatorname{Cox}(\overline{M}_{0,n})$ is an invariant subring of a $\operatorname{Pic}(\overline{{M}}_{0,n})$-graded polynomial ring and it is an intersection of two explicit finitely generated rings.
Fixing an isomorphism $\overline{M}_{0,n} \cong \operatorname{Bl}\mathbb{P}^{n-3}$ as in \cite{Kap,Has03} and a presentation of $\operatorname{Cox}(\overline{M}_{0,n})$ as in \cite{DG13},  
then \cite[Theorem 3.5]{DGJ} says that for any $d \geq 0$ there is a bijection between degree $d+1$ homogeneous elements of $\operatorname{Cox}(\overline{M}_{0,n})$ not divisible by any exceptional divisor section and weighted $d$-complexes in $\{1,2, \ldots , n-1 \}$ that satisfy a suitable balancing condition (see Definition \ref{definition.balanced}).
Under this bijection, all possible nondegenerate balancings on a balanceable $d$-complex $\Delta$ yield homogeneous elements of $\operatorname{Cox}(\overline{M}_{0,n})$ that are sections of the same effective divisor $D_{\Delta} \in \operatorname{Pic}(\overline{{M}}_{0,n})$ given by
\begin{equation*}   
D_\Delta = (d+1)H - \sum_I{\Bigg(d+1 - \max_{\sigma \in \Delta} \Bigg\{ \sum_{i\in I} m(i \in \sigma)\Bigg\} \Bigg)} E_I \in \operatorname{Pic}(\overline{{M}}_{0,n}),
\end{equation*}
where $H$ denotes the pullback of the hyperplane class of $\mathbb{P}^{n-3}$, $I \subseteq \{1, 2, \ldots , n-1\}$ satisfies $1 \leq |I| \leq n-4$ and $E_I$ denotes the exceptional divisor over the linear subspace of $\mathbb{P}^{n-3}$ spanned by the points in $I$, and $m(i \in \sigma)$ denotes the number of times $i$ appears in the multiset $\sigma$ (see Sections \ref{section.complexes} and \ref{section.simplicial.approach} for details).
%


 \subsection*{Structure of the article}

 In Section~\ref{section.complexes}, we review the language of weighted simplicial complexes from \cite{DGJ} and establish some basic properties of balancings, products and the \emph{link} construction. 
In Theorem~\ref{overall} we show that the balancing condition on a weighted simplicial complex from \cite{DGJ} holds overall if it holds for the facets.
In Section~\ref{section.simplicial.approach}, we review the main results from \cite{DGJ} about the simplicial approach to effective divisors on $\overline{M}_{0,n}$ and give several examples
for $n=7$, $n=8$, and for large values of $n$, arising from triangulations of $d$-dimensional tori.
In the case of nonsingular complexes, the known results are stronger (see Theorem~\ref{thm.simplicial.approach} for details).
With a view toward a systematic study of minimal balanced complexes, in Section~\ref{section.balancings} 
we study the vector spaces of balancings over a characteristic zero field of the complete $d$-complex $\Delta_{n,d}$ and the complete nonsingular $d$-complex $\Delta^{ns}_{n,d}$. In particular, we compute their dimensions and give explicit bases. 
We deduce as Corollary~\ref{corollary.no.nonsingular} that there are no nonsingular balanceable $d$-complexes on $n$ vertices when $d+1 > \frac{n}{2}$.  
As an application, in Example~\ref{example.nonsingular.6vertices} we list all irreducible elements of $M(\overline{M}_{0,7})$ that arise from nonsingular complexes. 
Hence, the main task left is to find a method to determine the irreducibility of elements in $M(\overline{M}_{0,7})$ arising from singular complexes (notice Theorem~\ref{thm.simplicial.approach} no longer gives irreducibility).
Section~\ref{section.theorem} is a worked out example of how to deal with this issue, or at least, how to deal with it inductively (on the degree of the complex).
We consider the \emph{hypertree divisor} $\mathcal{H}$ on $\overline{M}_{0,7}$ associated to the unique hypertree graph on seven vertices \raisebox{-0.5mm}{{\includegraphics[angle=0,origin=c, height=\myheight]{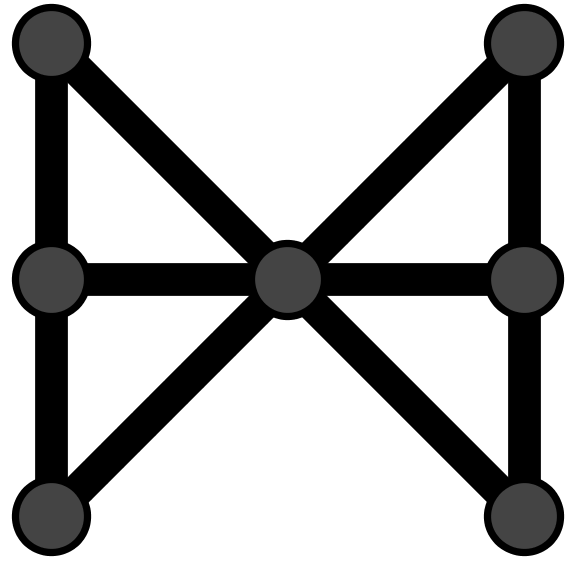}}}
which is an effective divisor irreducible in $M(\overline{M}_{0,7})$ by \cite[Theorem 1.5]{Hypertrees}.
The divisor $\mathcal{H}$ admits two representations as a complex in the different Kapranov models of $\overline{M}_{0,7}$, one of them by a nonsingular complex and one by a singular complex $\mathcal{A}$.
In Theorem~\ref{theorem.hypertree} we study the singular complex $\mathcal{A}$ and prove some known facts about the divisor $D_{\mathcal{A}}$ including its irreducibility in $M(\overline{M}_{0,7})$ (see Remark~\ref{remark.credit}).
We hope these ideas to show irreducibility are a step forward for the problem of determining the irreducibility in $M(\overline{M}_{0,7})$ of a given divisor arising from a singular complex, and hence toward an answer to Question 1.
%


\subsection*{Remark on our coefficients} Throughout, $R$ denotes a commutative ring with unit. 
We sometimes assume that $R$ is a characteristic zero domain or a characteristic zero field $R=K$.
For readers interested in more general coefficients, we point out that for a fixed $d \in \mathbb{Z}_{\geq 0}$ all results below about $d$-complexes and their proofs as presented, hold if one replaces \emph{characteristic zero domain} or \emph{characteristic zero field} with respectively a \emph{ring} or a \emph{field}, where $(d+1)!$ is not zero or a zero divisor, where these assumptions are present.


\subsection*{Acknowledgements} 
We would like to thank Erin Emerson, Connor Halleck-Dub\'e, Dave Jensen, (Jocelyn) Yuxing Wang and Nicholas Wawrykow for helpful discussions. 
We sincerely thank Jeremy Usatine for many insightful conversations.
This project started at the program Summer Undergraduate Mathematical Research at Yale (SUMRY).
We thank SUMRY and its organizer Sam Payne for their support.



\section{Weighted simplicial complexes}   \label{section.complexes}

In this section we review the language of weighted simplicial complexes from \cite{DGJ} and establish some basic properties of balancings, products and the link construction. 
Notably, in Theorem~\ref{overall} we show that the balancing condition on a weighted simplicial complex from \cite{DGJ} holds overall if it holds for the facets.
In Section \ref{section.simplicial.approach} we will review the connection of this material with effective divisors on the moduli space $\overline{M}_{0,n}$.


\subsection{Simplices and complexes}

Given a nonnegative integer $d$ and a set $\mathcal{S}$, 
a \emph{$d$-simplex} $\sigma$ on $\mathcal{S}$ 
is a multiset of cardinality $d+1$ whose elements are in $\mathcal{S}$
and a \emph{$d$-complex} $\Delta$ on $\mathcal{S}$ is a set $\Delta = \left \{ \sigma_1, \sigma_2, \ldots, \sigma_k\right\}$ whose elements are $d$-simplices on $\mathcal{S}$.
For simplicity, we often refer to these as \emph{simplices} and \emph{complexes}.
As an example, notice that a $1$-complex is the same as a graph with loops allowed but multiple edges disallowed.

\subsubsection{Multiplicities and nonsingularity}
The multiplicity of any $i \in \mathcal{S}$ in a multiset $S$ on $\mathcal{S}$ is by definition the number of times $i$ appears in the multiset $S$, and will be denoted in what follows by $m(i \in S)$.
A multiset $T$ is contained in a multiset $S$ if $m(i \in T) \leq m(i \in S)$ for every $i \in T$, and in that case the multiplicity $m(T \subseteq S)$ of $T$ in $S$ is defined to be
\begin{equation}  \label{definition.multiplicity}
	m(T \subseteq S) 
	= \prod_{i \in \mathcal{S}}{\binom{m(i\in S)}{m(i \in T)}}
		= \prod_{i\in \operatorname{Supp}(T)}{\binom{m(i\in S)}{m(i \in T)}}.
	\end{equation}
	
Note that the product in ($\ref{definition.multiplicity}$) is indeed identical if taken over $i \in \mathcal{S}$ or over the set of distinct elements in $T$ (which we denote by $\operatorname{Supp}(T)$, see \ref{subsection.supports.vertices}).
If $T \nsubseteq S$, the multiplicity $m(T \subseteq S)$ of $T$ in $S$ is defined to be zero and we observe that equation ($\ref{definition.multiplicity}$) still holds in this case. 
We use the symbol $\uplus$ to denote \emph{multiset sum}. 
For any multisets $S$ and $T$ on a set $\mathcal{S}$ and any $i \in \mathcal{S}$, $m( i \in S \uplus T ) = m( i \in S ) + m( i \in T )$.
The following lemma describes the behavior of the multiplicity $m(T \subseteq S)$ when a fixed multiset is added to both $T$ and $S$.

\begin{lemma}    \label{lemma.multiplicities}
Let $S$, $T$ and $U$ be multisets on a set $\mathcal{S}$. Then
\begin{equation*}
m(T \subseteq  S)  \cdot  m( S \subseteq S \uplus U )= m(T  \subseteq T \uplus U) \cdot m(T  \uplus U \subseteq S \uplus U).
\end{equation*}
\end{lemma}
\begin{proof}
For any $m,n,k \in \mathbb{Z}_{\geq 0}$ we have $\binom{m}{n}\binom{m+k}{m}=\binom{n+k}{n}\binom{m+k}{n+k}$. Hence, for any $i \in \mathcal{S}$ 
we have
\[
\binom{m(i \in S)}{m(i \in T)}\binom{m(i \in S)+m(i \in U)}{m(i \in S)}
=
\binom{m(i \in T)+m(i \in U)}{m(i \in T)}\binom{m(i \in S)+m(i \in U)}{m(i \in T)+m(i \in U)}.
\]
The desired formula now follows by taking the product of these identities over $i \in \mathcal{S}$. 
\end{proof}

A $d$-simplex $\sigma$ is \emph{nonsingular} if each of its elements has multiplicity one, and otherwise $\sigma$ is \emph{singular}.
A $d$-complex $\Delta$ is \emph{singular} if at least one of the simplices it contains is singular, and otherwise $\Delta$ is \emph{nonsingular}.

\subsubsection{Supports and vertices}                  \label{subsection.supports.vertices}
The \emph{support} $\operatorname{Supp}(\sigma)$ of a simplex $\sigma$ is the set of elements appearing in $\sigma$ disregarding their multiplicities.
The \emph{support} $\operatorname{Supp}(\Delta)$ of a complex $\Delta$ is the union of the supports of its simplices.
We will refer to the elements of $\mathcal{S}$ as \emph{vertices}. 
By the \emph{vertices} of a simplex or of a complex, we will mean the elements in its support.   
By \emph{a simplex or a complex on $n$ vertices}, we will mean respectively a simplex or a complex on some set $\mathcal{S}$ with cardinality $n$.
A $d$-complex $\Delta_1$ is a \emph{subcomplex} of a $d$-complex $\Delta_2$ if $\Delta_1$ is a subset of $\Delta_2$, and it is a \emph{proper subcomplex} if in addition $\Delta_1 \neq \Delta_2$.

\subsubsection{Product of complexes}
Given a $d_1$-complex $\Delta_1$ and a $d_2$-complex $\Delta_2$, their \emph{product} is the $(d_1+d_2 +1)$-complex
$\Delta_1 \cdot \Delta_2 = \{\sigma_1 \uplus \sigma_2 \:|\: \sigma_1 \in \Delta_1, \sigma_2 \in \Delta_2  \}$.

The following is an elementary but useful criterion for a complex not to be a product.

\begin{proposition} \label{nonproduct}
If for each vertex $i$ of a nonempty nonsingular complex $\Delta$ we have 
$$| \{ j \in \operatorname{Supp}(\Delta) :  \textnormal{$j\neq i$ and $\exists \sigma \in \Delta$ s.t. $i,j \in \sigma$} \} | < \frac{1}{2}{|\operatorname{Supp}(\Delta)|}$$ 
\emph{(}that is, each vertex shares a simplex with less than $\frac{1}{2}{|\operatorname{Supp}(\Delta)|}$ other vertices\emph{)},
then $\Delta$ is not a product. 
\end{proposition}
\begin{proof}
Suppose that $\Delta$ is the product $\Delta = \Delta_1 \cdot \Delta_2$ of two complexes $\Delta_1$ and $\Delta_2$. 
Since $\Delta$ is nonsingular, $\operatorname{Supp}(\Delta)$ is the disjoint union of $\operatorname{Supp}(\Delta_1)$ and $\operatorname{Supp}(\Delta_2)$.  
Without loss of generality, we may assume that $|\operatorname{Supp}(\Delta_1)| \geq |\operatorname{Supp}(\Delta_2)|$, and therefore $|\operatorname{Supp}(\Delta_1)| \geq |\operatorname{Supp}(\Delta)|/2$. Since any $j \in \operatorname{Supp}(\Delta_2) \neq \emptyset$ must share a simplex with every vertex in $\operatorname{Supp}(\Delta_1)$, we have a contradiction.
\end{proof}


\subsection{Weightings and balancings}

Given a ring $R$ and a complex $\Delta$, an $R$-\emph{weighting} on $\Delta$ is a function $w:\Delta \rightarrow R$. 
An $R$-\emph{weighted complex} $(\Delta,w)$ is a complex $\Delta$ endowed with an $R$-weighting $w$.
The set of $R$-weightings on a complex $\Delta$ has a natural $R$-module structure.
We will say that an $R$-weighting $w:\Delta \rightarrow R$ is \emph{nondegenerate} if $w(\sigma)\neq 0$ for all $\sigma \in \Delta$, and otherwise we say it is \emph{degenerate}.
This terminology slightly differs from that in \cite{DGJ} where its authors included the nondegeneracy condition in the definition of an $R$-weighting.
Given any $R$-weighted complex $(\Delta,w)$ on $\{1,2,\ldots,n\}$, we can associate the polynomial 
\[
P(\Delta,w) = \sum_{\sigma \in \Delta} w(\sigma)  \, x^{\sigma} \in R[ x_1,x_2, \ldots , x_n], \textnormal{ where } x^{\sigma} = \prod_{i \in \operatorname{Supp}(\sigma)}x_{i}^{m(i\in \sigma)},  \textnormal{ for each } \sigma \in \Delta.
\]
This assignment gives a one-to-one correspondence between nonempty $d$-complexes with a nondegenerate $R$-weighting and homogeneous polynomials of degree $d+1$ in $R[ x_1,x_2, \ldots , x_n]$.

\subsubsection{Balancing condition}   \label{definition.balanced}
A weighted $d$-complex $(\Delta,w)$ is \emph{balanced in degree} $j$ if for each multiset $S$ with $|S| =j$ we have
\[
\sum_{S \subseteq  \sigma  \in \Delta  }{w(\sigma) \cdot m(S \subseteq \sigma)} =0.
\]
A weighted $d$-complex is \emph{balanced} if it is balanced in degree $j$, for each $j \in \{ 0, 1, \ldots , d\}$.
A \emph{balancing} on a $d$-complex is a weighting that is balanced. 
We say that a $d$-complex is \emph{balanceable} if it admits a nondegenerate balancing (notice that every nonempty $d$-complex admits degenerate balancings).
The balancing condition was introduced in \cite[Definition 3.2]{DGJ} and it can be reinterpreted as follows. A weighted $d$-complex $(\Delta,w)$ is balanced if and only if the associated polynomial $P(\Delta,w) \in R[ x_1,x_2, \ldots , x_n] \subseteq R[ x_1,x_2, \ldots , x_n,y]$ is invariant under the $R$-algebra automorphism $\phi:R[ x_1,x_2, \ldots , x_n,y]\rightarrow R[ x_1,x_2, \ldots , x_n,y]$ determined by $\phi(y)=y$ and $\phi(x_i)=x_i+y$ for all $i=1,2,\ldots,n$, where $y$ is an independent variable.

\begin{example}      \label{example.balancing.product}
If $(\Delta_1,w_1)$ and $(\Delta_2,w_2)$ are balanced weighted complexes, their product $\Delta = \Delta_1 \cdot \Delta_2$ inherits a balanced weighting $w$ which for each $\sigma \in \Delta$ is defined by 
\[
w(\sigma) = \sum_{\substack{\sigma_1 \in \Delta_1, \sigma_2 \in \Delta_2  \\  \sigma = \sigma_1 \uplus \sigma_2  }}   w_1(\sigma_1) \, w_2(\sigma_2).
\]
By construction, the weighted complex $(\Delta,w)$ corresponds to the polynomial $P(\Delta,w)=P(\Delta_1,w_1)P(\Delta_2,w_2)$, so it is indeed balanced as the product of invariant polynomials under $\phi$ is again invariant. 
Notice that $w$ might be degenerate even if $w_1$ and $w_2$ are nondegenerate.
\end{example}

\subsubsection{Minimal complexes}
A $d$-complex $\Delta$ is \emph{minimal} if it admits a nondegenerate balancing but none of its nonempty proper subcomplexes admits a nondegenerate balancing.

\begin{proposition}[cf. {\cite[Proposition 3.20]{DGJ}}] \label{onedimvectorspace}
Let $\Delta$ be a nonempty $d$-complex that admits a nondegenerate balancing over a field $R=K$. Then $\Delta$ is minimal if and only if 
the vector space of $K$-balancings of $\Delta$ is one-dimensional. 
\end{proposition}
\begin{proof}
Suppose that $\Delta = \{\sigma_{1},\sigma_2,\ldots,\sigma_{r}   \}$ and that $w:\Delta \rightarrow K$ is a nondegenerate balancing of $\Delta$. 
We may assume that $r \geq 2$. Let us prove the contrapositive of the give assertion.
If $\Delta$ admits a $K$-balancing $w':\Delta \rightarrow K$ linearly independent from $w$, then $w'(\sigma_1)w-w(\sigma_1)w'$ is a nondegenerate balancing of a nonempty proper subcomplex of $\Delta$, so $\Delta$ would not be minimal.
Conversely, if $\Delta$ has a nonempty proper subcomplex $\Delta'$ admitting a nondegenerate balancing $w':\Delta' \rightarrow K$, then $w'$ can be extended to a balancing $w':\Delta \rightarrow K$ that takes the value zero on $\Delta \smallsetminus \Delta'$ and which is clearly linearly independent from $\Delta$, thus the space of $K$-balancings of $\Delta$ is at least two dimensional.
\end{proof}


\subsection{Facet balancing implies overall balancing}

For any positive integer $d$, a \emph{facet} of a $d$-simplex $\sigma$ is a $(d-1)$-simplex contained in $\sigma$. 
By definition, the only facet of a $0$-simplex is the empty set.
A facet of a complex is a facet of one of its simplices.
Note that each nonsingular $d$-simplex contains exactly $d+1$ facets.

\begin{theorem}  \label{overall}
Let $\Delta$ be a weighted $d$-complex with weights on a characteristic zero domain $R$. If $\Delta$ is balanced in degree $d$, i.e. balanced for each of its facets, then it is balanced. 
\end{theorem}
\begin{proof}
We may assume that $d \geq 1$ and that $\Delta = \left \{ \sigma _1, \sigma _2, \ldots, \sigma_r \right \}$ is a $d$-complex on $\{ 1, 2, \ldots,n \}$.
We denote the respective weights of the simplices in $\Delta$ by $w_1,w_2, \ldots,w_r$.
Let us assume that $\Delta$ is balanced in degree $d_0$ for some $1 \leq d_0 \leq d$ and 
we let $T \subseteq  \{ 1, 2, \ldots,n \}$ be any multiset with cardinality $|T| = d_0-1$. 
We denote $T_i = T \uplus \{ i \}$, $a_{ij} = m(i \in \sigma_j)$ and $b_i = m(i \in T)$ for each $i \in \{ 1, 2, \ldots,n \}$ and $j \in \{ 1, 2, \ldots,r \}$.
Using the identity $(n+1) {\binom{m}{n+1}} =(m-n){\binom{m}{n}}$ for any $m,n \in \mathbb{Z}_{\geq 0}$, as well as equation $(\ref{definition.multiplicity})$, for a fixed $1 \leq j \leq r$ we get

\begin{align*}
 \sum_{i =1}^{n}{m(T \subseteq T_i) \cdot m(T_i \subseteq \sigma_j)}  
&= \sum_{i =1}^{n}{(b_i+1) \cdot \binom{a_{ij}}{b_i +1} \cdot \prod_{\substack {1 \leq k \leq n \\ k \neq i}}{\binom{a_{kj}}{b_k}} }  
= \sum_{i =1}^{n}{(a_{ij}-b_i ) \cdot  \binom{a_{ij}}{b_i} \cdot \prod_{\substack {1 \leq k \leq n \\ k \neq i}}{\binom{a_{kj}}{b_k}} }
\\
&=  \sum_{i =1}^{n}{(a_{ij}-b_i ) \cdot m(T \subseteq \sigma_j)}  
= (|\sigma_j|-|T|) \cdot m(T \subseteq \sigma_j)
\\
&= (d-d_0+2) \cdot m(T \subseteq \sigma_j).
\end{align*}
By assumption $\Delta$ satisfies the balancing condition for each of the multisets $T_1,T_2,\ldots,T_n$, and hence we have
\begin{align*} 0 &=\sum_{i = 1}^{n}{m(T\subseteq T_i) \sum_{T_i \subseteq \sigma_j}{w_j \cdot m(T_i \subseteq \sigma_j)}}
=\sum_{T \subseteq \sigma_j}{w_j \cdot \sum_{T_i \subseteq \sigma_j }{m(T \subseteq T_i) \cdot m(T_i \subseteq \sigma_j)}}\\
&=\sum_{T \subseteq \sigma_j}{w_j \cdot \sum_{i =1}^{n}{m(T \subseteq T_i) \cdot m(T_i \subseteq \sigma_j)}}
=(d-d_0+2) \cdot \sum_{T \subseteq \sigma_j}{w_j} \cdot  m(T \subseteq \sigma_j) 
\end{align*}

Since $d-d_0+2$ is not zero or a zero divisor, we conclude that
\[
\sum_{T \subseteq \sigma_j}{w_j \cdot m(T \subseteq \sigma_j)}=0.
\]
It follows that $\Delta$ is balanced in degree $d_0-1$ if it is balanced in degree $d_0$. Since $\Delta$ is balanced in degree $d$ by assumption, 
using this argument repeatedly gives that $\Delta$ is balanced, as it has to be balanced in all required degrees $0 \leq d_0 \leq d$.
 \end{proof}


 \subsection{The link construction} Given a $d$-complex $\Delta$ and given $i \in \operatorname{Supp}(\Delta)$, 
 we define an associated $(d-1)$-complex $\Delta^*_{i}$ called the \emph{link} of $\Delta$ with respect to the vertex $i$ by 
 \[
 \Delta^{*}_{i}= \{  \sigma   \   |   \  \sigma \uplus \{i\}  \in \Delta  \}.
 \]
In other words, the simplices in $\Delta^*_{i}$ are obtained by taking the simplices in $\Delta$ containing $i$ and reducing its multiplicity by $1$.
More generally, if $S$ is a multiset of size $k \leq d$ and we define the $(d-k)$-complex $\Delta^*_{S}$ called the \emph{link} of $\Delta$ with respect to $S$ by 
 \[
 \Delta^{*}_{S}= \{  \sigma   \   |   \  \sigma \uplus S  \in \Delta  \}.
 \]
Any such link construction is equivalent to a particular iteration of the link construction on single vertices. 
In this subsection we will relate some properties of a given complex with those of its links.

\begin{example} \label{link.product} 
If a nonsingular complex $\Delta$ is a product  
$\Delta = \Delta_1 \cdot \Delta_2$, then 
$\Delta_i^* = (\Delta_1)_i^* \cdot \Delta_2$ for any vertex $i$ of $\Delta_1$
and
$\Delta_j^* = \Delta_1 \cdot (\Delta_2)_j^*$ for any vertex $j$ of $\Delta_2$. 
\end{example}

Any weighting $w: \Delta \rightarrow R$ on $\Delta$ induces a weighting $w^{*}_{S}:\Delta^{*}_{S} \rightarrow R$ on $w^{*}_{S}$, defined on each $\sigma$ by the formula
\[
w^{*}_{S}(\sigma)=w(\sigma \uplus S ) \cdot m(\sigma \subseteq \sigma \uplus S ).
\]
In the case that $S=\{i \}$ this formula becomes $w^{*}_{i}(\sigma)=w(\sigma \uplus \{i \} ) \cdot m(i \in \sigma \uplus \{i \} )$.

We observe that over a characteristic zero domain $R$, if the weighting $w:\Delta \rightarrow R$ on the $d$-complex $\Delta$ is nondegenerate, then the weighting $w^{*}_{S}:\Delta^{*}_{S} \rightarrow R$ is also nondegenerate. 

\begin{proposition}
If the $d$-complex $\Delta$ is balanced by the weight function $w:\Delta \rightarrow R$, then for any multiset $S$ with $|S| \leq d$
the complex $\Delta^{*}_{S}$ is balanced by the weight function $w^{*}_{S}:\Delta^{*}_{S} \rightarrow R$.
\end{proposition}
\begin{proof}
Let us fix a multiset $T$ on $\{ 1, 2, \ldots , n\}$ with $|T| \leq d-|S|$.
By Lemma~\ref{lemma.multiplicities} we have
\begin{equation*}
m(\sigma \subseteq \sigma \uplus S)  \cdot  m(T \subseteq \sigma)=m(T \uplus S \subseteq \sigma \uplus S) \cdot m(T \subseteq T \uplus S).
\end{equation*}
Therefore, the balancing condition on $\Delta^{*}_{S}$ for $T$ is implied by that on $\Delta$ for $T \uplus S$ as follows,
\begin{align*} 
\sum_{T \subseteq \sigma \in \Delta^{*}_{S}} w^{*}_{S}(\sigma) \cdot m(T \subseteq \sigma)
&=\sum_{T \subseteq \sigma \in \Delta^{*}_{S}} w(\sigma \uplus S ) \cdot m(\sigma \subseteq \sigma \uplus S ) \cdot m(T \subseteq \sigma)
 \\
&=\sum_{T \subseteq \sigma \in \Delta^{*}_{S}} w(\sigma \uplus S ) \cdot m(T \uplus S \subseteq \sigma \uplus S) \cdot m(T \subseteq T \uplus S)
\\
&= m(T \subseteq T \uplus S) \cdot \sum_{T \uplus S \subseteq \sigma  \uplus S \in \Delta} w(\sigma \uplus S )  \cdot m(T \uplus S \subseteq \sigma \uplus S)
=0.
\end{align*}
\end{proof}

\begin{example}  
If $\Delta$ is a non-minimal $d$-complex over a characteristic zero domain $R$,
then either  
$\Delta$ is the union of two nonempty complexes on disjoint sets of vertices or 
there exist a vertex $i$ such that $\Delta_i^*$ is non-minimal.
\end{example}

Conversely, the balancing condition for a weighting $w$ on $\Delta$ is implied by the balancing of the proper links $\Delta^{*}_{S}$ by the weightings $w^{*}_{S}$ as we see next.   

\begin{proposition}     \label{proposition.links.balanced}
Let $\Delta$ be a $d$-complex with a weighting $w: \Delta \rightarrow R$ over a characteristic zero domain $R$, with $d \geq 1$.
If the induced weightings $w^{*}_{i}$ on the links $\Delta^{*}_{i}$ of all vertices are balanced, then $w$ is balanced. 
\end{proposition}
\begin{proof}
Let $S$ be a multiset on $\operatorname{Supp}(\Delta)$ with $1 \leq |S| \leq d$. 
Fix a vertex $i \in S$ and let $\bar{S}$ be the multiset such that $S=\bar{S}\uplus \{i\}$.
For any $\sigma \in \Delta^*_{i}$, using Lemma~\ref{lemma.multiplicities} we have
\begin{align}  \label{equation.multiplicities.vertex}
m(\{i\} \in  \sigma \uplus \{i\})  \cdot  m( \bar{S} \subseteq \sigma) 
&=m(\sigma \subseteq   \sigma \uplus \{i \})  \cdot  m( \bar{S} \subseteq \sigma )     
= m( \bar{S}  \subseteq S) \cdot m( S  \subseteq \sigma \uplus \{i \})    \nonumber  \\
&= m(\{i \} \in  S) \cdot m(S  \subseteq \sigma \uplus \{i \}).
\end{align}
Using (\ref{equation.multiplicities.vertex}) and the balancing condition for $\Delta^{*}_{i}$ for $\bar{S}$ we have 
\begin{align*} 
0&=\sum_{\bar{S} \subseteq \sigma \in \Delta^{*}_{i}} w^{*}_{i}(\sigma) \cdot m(\bar{S} \subseteq \sigma)
=\sum_{\bar{S} \subseteq \sigma \in \Delta^{*}_{i}} w(\sigma \uplus \{i \} ) \cdot 
m(\{i\} \in  \sigma \uplus \{i\})  \cdot  m( \bar{S} \subseteq \sigma)    
\\
&=\sum_{\bar{S} \subseteq \sigma \in \Delta^{*}_{i}}  w(\sigma \uplus \{i \} ) \cdot
m(\{i \} \in  S) \cdot m(S  \subseteq \sigma \uplus \{i \})   
= m(\{i \} \in  S) \cdot  \sum_{\bar{S} \subseteq \sigma \in \Delta^{*}_{i}}  w(\sigma \uplus \{i \} ) \cdot
m(S  \subseteq \sigma \uplus \{i \})   
\\
&=m(\{i \} \in  S) \cdot  \sum_{S \subseteq \tau \in \Delta}  w(\tau ) \cdot  m(S  \subseteq \tau). 
\end{align*}
Since $m(\{i \} \in  S)$ is not zero or a zero divisor the balancing condition for $S$ on $\Delta$ holds. Theorem~\ref{overall} implies that $w$ is a balancing for $\Delta$ as desired.
\end{proof}


\begin{example}    \label{example.links.balancing}
The complex $\Delta=\{\{1,2\},\{1,3\},\{2,3\}\}$ does not admit any non-identically zero balancings over a characteristic zero domain, even though all its proper links admit nondegenerate balancings.
\end{example}


Now we relate the existence of nondegenerate balancings for a nonsingular product and each of the factors.

\begin{proposition}         \label{why}
Let $\Delta$ be a nonempty nonsingular complex that is a product.
If $\Delta$ admits a nondegenerate balancing over a ring $R$ then both factors admit nondegenerate balancings over $R$.
The converse holds if $R$ is a domain.
\end{proposition}
\begin{proof}
Suppose $\Delta= \Delta_1 \cdot \Delta_2$ for some $d_1$-complex $\Delta_1$ and some $d_2$-complex $\Delta_2$.
Then $\Delta_1$ and $\Delta_2$ are nonsingular and their supports are disjoint.
By symmetry, it is enough to show $\Delta_1$ admits the desired balancing.
Let $w:\Delta\rightarrow R$ be a nondegenerate balancing for $\Delta$. 
Fix a simplex $\tau$ in $\Delta_2$ and define a nondegenerate weighting $w_{\tau}:\Delta_{1}\rightarrow R$ by $w_{\tau}(\sigma)=w(\sigma \uplus \tau)$.
The balancing condition for any multiset $S$ on $\operatorname{Supp}(\Delta_1)$ with $|S| \leq d_{1}$ is verified as follows
 \begin{align*}
\sum_{S \subseteq \sigma \in \Delta_1} w_{\tau}(\sigma) \cdot m(S \subseteq \sigma) = \sum_{S\uplus \tau \subseteq \sigma \uplus \tau \in \Delta} w(\sigma \uplus \tau) \cdot m(S \uplus \tau \subseteq \sigma \uplus \tau)= \sum_{S\uplus \tau \subseteq \lambda \in \Delta} w(\lambda) \cdot m(S \uplus \tau \subseteq \lambda) =0.
\end{align*}
Conversely, given nondegenerate balancings $w_{1}:\Delta_{1}\rightarrow R$ and $w_{2}:\Delta_{2}\rightarrow R$, we define a weighting $w:\Delta \rightarrow R$ by $w(\sigma)=w_{1}(\sigma_1) \cdot w_{2}(\sigma_2)$, where $\sigma =\sigma_1 \uplus \sigma_2$ is the unique such expression with $\sigma_{1}\in \Delta_{1}$ and $\sigma_2 \in \Delta_{2}$. 
This weighting is nondegenerate if $R$ is domain. Given any multiset $S$ on $\operatorname{Supp}(\Delta)$ with $|S| \leq d_{1}+d_{2}+1$, it can be written in a unique way as $S=S_1 \uplus S_2$ for some multisets $S_1$ and $S_2$ respectively on $\operatorname{Supp}(\Delta_{1})$ and $\operatorname{Supp}(\Delta)_{2}$. Since either $|S_1| \leq d_{1}$ or $|S_2| \leq d_{2}$, we have
 \begin{align*}
\sum_{S \subseteq \sigma \in \Delta} \! w(\sigma) \cdot m(S \subseteq \sigma) 
&=\sum_{S \subseteq \sigma \in \Delta} \! w(\sigma)
=\sum_{\substack{S_1 \subseteq \sigma_1 \in \Delta_1    \\ S_2 \subseteq \sigma_2 \in \Delta_2}} \! w_1(\sigma_1) \cdot w_2(\sigma_2)     
=  \left( \sum_{S_1 \subseteq \sigma_1 \in \Delta_1} \! w(\sigma_1)  \right) \cdot  \left( \sum_{S_2 \subseteq \sigma_2 \in \Delta_2} \! w(\sigma_2)  \right)
\\
&= \left( \sum_{S_1 \subseteq \sigma_1 \in \Delta_1} w(\sigma_1) \cdot m(S_1 \subseteq \sigma_1) \right) \cdot  \left( \sum_{S_2 \subseteq \sigma_2 \in \Delta_2} w(\sigma_2) \cdot m(S_2 \subseteq \sigma_2)  \right)  =  0.
\end{align*} 
\end{proof}



\section{A simplicial approach to effective divisors on $\overline{M}_{0,n}$}     \label{section.simplicial.approach}

In the remainder of this article we think of the moduli space $\overline{M}_{0,n}$ using Kapranov-Hassett's description in \cite{Kap,Has03}.   
In other words, we fix points $p_1, p_2,\ldots,p_{n-1} \in \mathbb{P}^{n-3}$ in linearly general position and identify $\overline{M}_{0,n}$ with the iterated blow up of $\mathbb{P}^{n-3}$ along all linear subspaces spanned by subsets $\{ p_i \}_{i \in I}$ of these points 
with $1\leq |I| \leq n-4 $ 
(more precisely, the blow up is taken in a suitable order and along the strict transforms of the subspaces). 
We denote by $H$ the pullback of the hyperplane class of $\mathbb{P}^{n-3}$
and by $E_{I}$ the strict transform of the exceptional divisor that projects onto the subspace spanned by the points $p_i$ with $i \in I$.  
The classes $H$ and $E_{I}$, for each $I \subseteq \{1,2,\ldots, n-1\}$ with $1\leq |I| \leq n-4 $, form a $\mathbb{Z}$-basis of $\operatorname{Pic}(\overline{{M}}_{0,n})$.
Our motivation to study weighted simplicial complexes in Section~\ref{section.complexes} is the connection to effective divisors on $\overline{M}_{0,n}$ summarized in the following theorem of Doran, Giansiracusa and Jensen (see~\cite[Theorem 1.3 and Theorem 1.4]{DGJ}).
\begin{thmx}[Doran-Giansiracusa-Jensen's simplicial approach to effective divisors on $\overline{M}_{0,n}$]     \label{thm.simplicial.approach}
Let $d \geq 0$ and $n \geq 5$ be positive integers. Then,
\begin{itemize}
\item[(A1)] There is a bijection between degree $d + 1$ multihomogeneous
elements of $\operatorname{Cox}(\overline{M}_{0,n})$, not divisible by any exceptional divisor section, and nondegenerately balanced
$d$-complexes on $\{1,2,\ldots, n-1\}$. Moreover, all nondegenerate balancings on a complex $\Delta$ correspond to elements with class 
\begin{equation}   \label{DGJrecipe2}
D_\Delta = (d+1)H - \sum_I{\Bigg(d+1 - \max_{\sigma \in \Delta} \Bigg\{ \sum_{i\in I} m(i \in \sigma)\Bigg\} \Bigg)} E_I \in \operatorname{Pic}(\overline{{M}}_{0,n}),
\end{equation}
\item[(A2)] Let $\Delta$ be a nonsingular, balanceable, minimal $d$-complex on $\{1,2,\ldots, n-1\}$, with $d \leq n - 5$, over a field, and which is not a product. 
Then $D_{\Delta}$ is irreducible in $M(\overline{M}_{0,n})$, $h^{0}(\overline{M}_{0,n},D_{\Delta})$=1, and every generating set for
$\operatorname{Cox}(\overline{M}_{0,n})$ includes the unique up to scalar section of $D_{\Delta}$.
\end{itemize}
\end{thmx}

We now give a series of examples over a characteristic zero field $K$.

\begin{example}
For any positive integers $d \geq 0$ and $n \geq 5$ and any $d$-complex $\Delta$ admitting a nondegenerate balancing, the divisor $D_{\Delta}$ is effective on $\overline{M}_{0,n}$.  
Moreover, 
if the class $D \in \operatorname{Pic}(\overline{M}_{0,n})$ is such that $D - E_{I}$ is not effective for any exceptional divisors $E_{I}$, then $D$ is effective if and only if there is a complex $\Delta$, admitting a nondegenerate balancing, such that $D_{\Delta} = D$.
\end{example}

\begin{example} \label{example.degree.zero}
In the Kapranov model of $\overline{M}_{0,7}$ with respect to the seventh marked point the irreducible elements of degree zero 
in the monoid of effective divisors
are the exceptional classes $E_{i}, E_{ij}, E_{ijk}$ where $i,j,k$ are distinct elements in $\{1,2,\ldots,6\}$. There are ${6 \choose 1} + {6 \choose 2} + {6 \choose 3} =41$ of these classes.
\end{example}

\begin{example}    \label{example.degree.one}
In the Kapranov model of $\overline{M}_{0,7}$ with respect to the seventh marked point the irreducible elements of degree one 
in the monoid of effective divisors
are the  divisors associated to the $0$-complexes $B_{ij}=\left\{ \{i\}, \{j\} \right\}$ where $i,j$ are distinct elements in $\{1,2,\ldots,6\}$ (See Figure~\ref{fig-complexes}). There are 15 of these classes and together with the divisor classes in Example~\ref{example.degree.zero} they are the ${7 \choose 2} + {7 \choose 3}=56$ \emph{boundary divisors} of $\overline{M}_{0,7}$.
\end{example}

\begin{example}    \label{example.degree.two}
The case $n=7$ of \cite[Theorem 1.5]{DGJ}, which was proved as an application of Theorem~\ref{thm.simplicial.approach}, says that
in the Kapranov model of $\overline{M}_{0,7}$ with respect to the seventh marked point the irreducible elements of degree two 
in the monoid of effective divisors
are the  divisors associated to the $1$-complexes $H_{ijkpqr}$, $T_{(i)(jk)(pq)}$ and $P_{(i)(jk)(r)(pq)}$ defined by
\begin{align*}
H_{ijkpqr} &= \left\{ \{i,j\}, \{j,k\}, \{k,p\}, \{p,q\}, \{q,r\}, \{r,i\} \right\}         
\\
T_{(i)(jk)(pq)} &= \left\{ \{i,j\}, \{i,k\}, \{j,k\}, \{i,p\}, \{i,q\}, \{p,q\} \right\}    
\\
P_{(i)(jk)(r)(pq)} &=\left\{ \{i,j\}, \{i,k\}, \{j,k\}, \{r,p\}, \{r,q\}, \{p,q\}, \{i,r\} \right\}       
\end{align*}
where $i,j,k,p,q,r$ are distinct elements in $\{1,2,\ldots,6\}$ (See Figure~\ref{fig-complexes}). 
There are respectively 60, 90 and 90 of these divisors.
\end{example}
\begin{figure}[ht]     %
\begin{center}
\includegraphics[height=2.7cm]{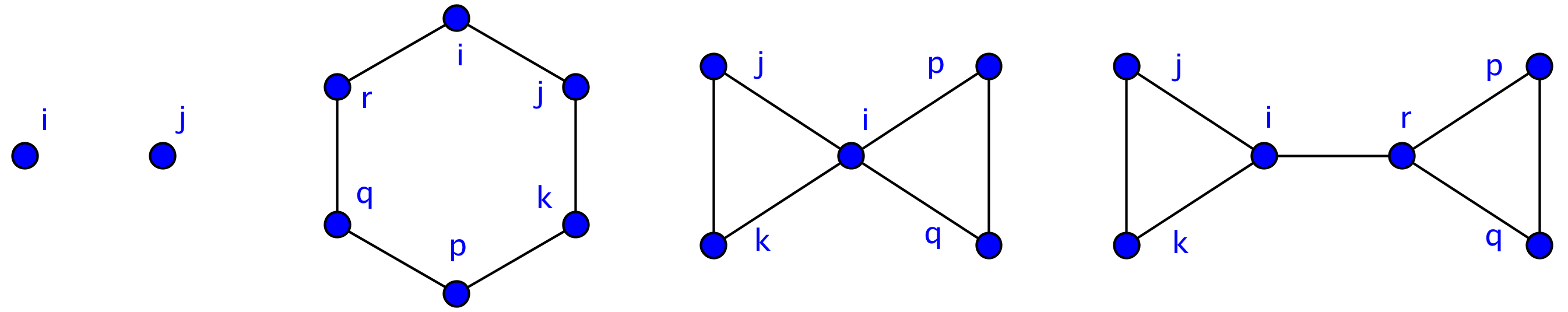} 
\caption{The complexes $B_{ij}$, $H_{ijkpqr}$, $T_{(i)(jk)(pq)}$ and $P_{(i)(jk)(r)(pq)}$.}    \label{fig-complexes}
\end{center}
\end{figure}

\begin{example}    \label{example.degree.three}
We wrote a fairly simple computer program on Python whose output is all minimal balanced nonsingular $2$-complexes on $\{1,2,\ldots,6\}$.
Up to symmetry, there are only two such complexes 
\begin{align*}
&\mathcal{O}_{(i_1i_5)(i_2i_6)(i_3i_4)}=\{\{i_1, i_2, i_3\}, \{i_1, i_2, i_4\}, \{i_1, i_3, i_6\}, \{i_1, i_4, i_6\}, \{i_2, i_3, i_5\}, \{i_2, i_4, i_5\}, 
\{i_3, i_5, i_6\}, \{i_4, i_5, i_6\}\},     \\
&\mathcal{C}_{(i_1i_4)(i_2i_3)(i_5i_6)}=\{\{i_1, i_2, i_3\}, \{i_1, i_2, i_4\}, \{i_1, i_3, i_4\}, \{i_1, i_4, i_5\}, \{i_1, i_4, i_6\}, \{i_1, i_5, i_6\}, \{i_2, i_3, i_4\}, \{i_2, i_3, i_5\},      \\ 
& \phantom{\mathcal{C}_{(i_1i_4)(i_2i_3)(i_5i_6)}=\{\{i_1, i_2, i_3\}, \{i_1, i_2, i_4\}, \{i_1, i_3, i_4\},  \{i_1, i_4, i_5\}, } 
\{i_2, i_3, i_6\}, \{i_2, i_5, i_6\}, \{i_3, i_5, i_6\}, \{i_4, i_5, i_6\}\},
\end{align*}
\end{example}
where $i_1,i_2,\ldots,i_6$ are the distinct elements in $\{1,2,\ldots,6\}$.
The complex $\mathcal{O}_{(i_1i_5)(i_2i_6)(i_3i_4)}$ is given by the faces of an octahedron with vertices suitably labeled (see Figure~\ref{fig-Octahedron}) and it is the product $\{i_1,i_5\}\cdot \{i_2,i_6\} \cdot \{i_3,i_4\}$, hence its associated effective divisor is not irreducible in $M(\overline{M}_{0,7})$.
The complex $\mathcal{C}_{(i_1i_4)(i_2i_3)(i_5i_6)}$ is given by the faces of three tetrahedra attached in a precise way as a cycle  (see Figure~\ref{fig-Cycle3Tetrahedra}). It is easy to see that the complex $\mathcal{C}_{(i_1i_4)(i_2i_3)(i_5i_6)}$ is nonsingular, balanceable, minimal and not a product (see \cite[4.2.4]{DGJ}). 
Hence, by Theorem~\ref{thm.simplicial.approach}, 
in the Kapranov model of $\overline{M}_{0,7}$ with respect to the seventh marked point, its associated divisor class $D_{\mathcal{C}_{(i_1i_4)(i_2i_3)(i_5i_6)}}$ is irreducible in the monoid $M(\overline{M}_{0,7})$. The divisor $D_{\mathcal{C}_{(i_1i_4)(i_2i_3)(i_5i_6)}}$ is the pullback to $\overline{M}_{0,7}$ of the unique hypertree divisor in $\overline{M}_{0,6}$. 
\begin{figure}[ht]     
\begin{center}
\begin{minipage}{0.45\linewidth}
\begin{center}
\includegraphics[height=3cm]{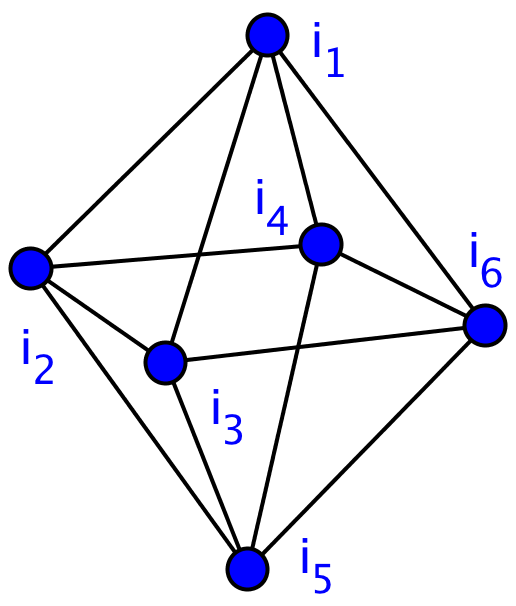} 
\end{center}
\caption{Complex $\mathcal{O}_{(i_1i_5)(i_2i_6)(i_3i_4)}$.}    \label{fig-Octahedron}
\end{minipage}
\begin{minipage}{0.45\linewidth}
\begin{center}
\includegraphics[height=3cm]{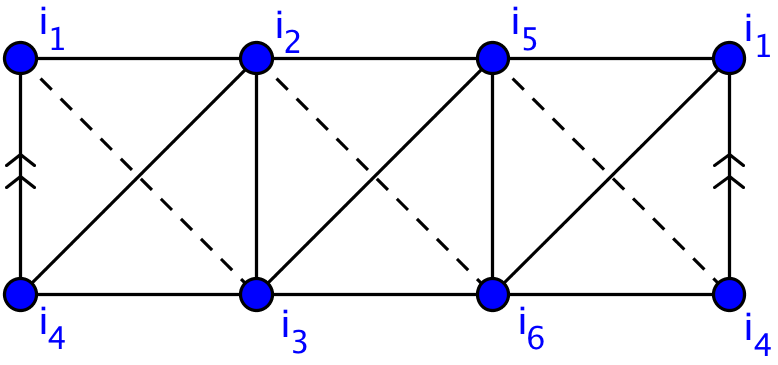}
\end{center}
\caption{Complex $\mathcal{C}_{(i_1i_4)(i_2i_3)(i_5i_6)}$.}   \label{fig-Cycle3Tetrahedra}
\end{minipage}
\end{center}
\end{figure}

\begin{example}       \label{example.nonsingular.6vertices}
By Corollary~\ref{corollary.no.nonsingular}, there are no balanceable nonsingular $d$-complexes on $6$ vertices for $d \geq 3$. 
From Examples~\ref{example.degree.one}-\ref{example.degree.three},  
we deduce that in the Kapranov model of $\overline{M}_{0,7}$ with respect to the seventh marked point,
the effective divisors that are irreducible in the monoid $M(\overline{M}_{0,7})$ and that arise from a nonsingular balanced complex are $B_{ij}$, $H_{ijkpqr}$, $T_{(i)(jk)(pq)}$, $P_{(i)(jk)(r)(pq)}$ and $\mathcal{C}_{(i_1i_4)(i_2i_3)(i_5i_6)}$ where the subindices of each complex are distinct elements of $\{1,2,\ldots,6\}$.
These are represented in Figures~\ref{fig-complexes} and \ref{fig-Cycle3Tetrahedra}.
\end{example}

\begin{example}\textbf{An effective divisor irreducible in $M(\overline{M}_{0,8})$.}
All triangulations of the real projective plane $\mathbb{RP}^2$ can be obtained by successive \emph{vertex splittings} from the triangulations $\Delta'$ and $\Delta''$ shown in Figures~\ref{fig-six-triangulation-RP2} and \ref{fig-seven-triangulation-RP2} (see \cite{TriangulationsRP2} for details). 
The triangulation $\Delta'$ is not balanceable.    
The dual graph of the triangulation $\Delta''$ is bipartite, hence as a 2-complex  $\Delta''$  is balanceable and minimal.  
Suppose that $\Delta''$ was the product of two lower degree complexes $\Delta_{0}$ and $\Delta_1$,
which we can assume have degree zero and one, respectively, and necessarily have disjoint supports.
Since the vertex $7$ is in at least one simplex with each of the other vertices, then $7 \in \operatorname{Supp}(\Delta_1)$.
By Example~\ref{link.product}, $({\Delta''})_7^* = (\Delta_0) \cdot (\Delta_1)^*_7$, but this is a contradiction since $(\Delta'')_7^*$ is a hexagon and hence not equal to the product of two $0$-complexes.
\begin{figure}[ht]     
\begin{center}
\begin{minipage}{0.45\linewidth}
\begin{center}
\includegraphics[height=3cm]{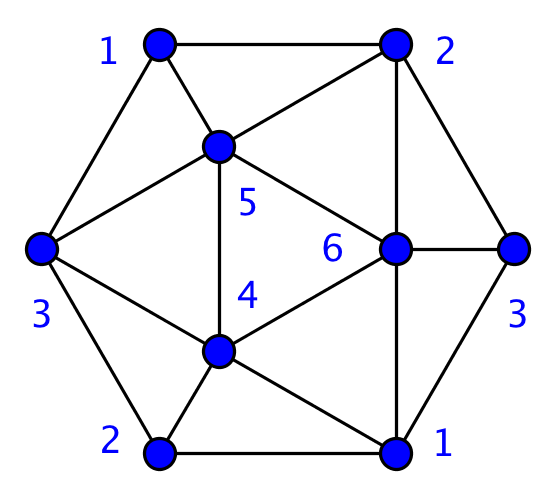} 
\end{center}
\caption{Triangulation $\Delta'$ of $\mathbb{RP}^2$.}    \label{fig-six-triangulation-RP2}
\end{minipage}
\begin{minipage}{0.45\linewidth}
\begin{center}
\includegraphics[height=3cm]{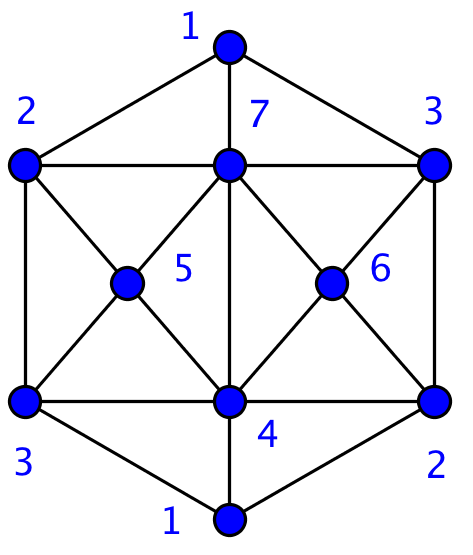}
\end{center}
\caption{Triangulation $\Delta''$ of $\mathbb{RP}^2$.}   \label{fig-seven-triangulation-RP2}
\end{minipage}
\end{center}
\end{figure}
Then, $\Delta''$ is nonsingular, balanceable, minimal and it is not a product. 
By Theorem~\ref{thm.simplicial.approach}, over a field, the associated divisor class $D_\Delta''$ given by the formula in (\ref{DGJrecipe2}), whose first terms are 
\[
D_\Delta'' = 
3H 
- \sum_{1\leq i \leq 7} 2E_{i}    
-2E_{15}-2E_{16}-2E_{56}
- \sum_{\substack{1\leq i < j \leq 7  \\  (i,j) \neq (1,5),(1,6),(5,6) }} E_{ij}
- \quad \cdots
\]
is irreducible in the monoid of effective divisors of $\overline{M}_{0,8}$ and every generating set for
$\operatorname{Cox}(\overline{M}_{0,8})$ includes its unique up to scalar associated section.
\end{example}


\begin{example}\textbf{Effective divisors irreducible in $M(\overline{M}_{0,n})$ from triangulated $d$-tori.}       \label{example.appendix}
Let $d\geq 2$ and $n_1,n_2,\ldots,n_d \geq 3$ be integers, such that $dn_1,dn_2,\ldots,dn_d$ are all even. Then, over a characteristic zero field, the divisor class $D_{n_1,n_2,\ldots,n_d}$ given by
\begin{equation}    \label{class.d.tori}
(d+1)H-\!\!\!\!\!
\sum\limits_{\substack{I \subseteq V\\1\leq |I| \leq n-4}}   \!\!
\left[     
(d+1)-\!\!\!
\max\limits_{\substack{(b_1,\ldots,b_d) \in V\\\sigma \in S_d}}
\left|
\{
(a_1,\ldots,a_d)\!\in\!I \ | \ 0\!\leq\!a_{\sigma(1)}\!-\!b_{\sigma(1)}\!\leq\!\ldots\!\leq\!a_{\sigma(d)}\!-\!b_{\sigma(d)}\!\leq\!1
\}   
\right|
\right]  E_{I}
\end{equation}
is irreducible in the monoid of effective divisors of $\overline{M}_{0,n}$, where $n=n_1n_2\cdots n_d+1$ and $V=[0,n_1-1]\times[0,n_2-1]\times \cdots \times[0,n_d-1]\cap \mathbb{Z}^d$. Moreover $h^{0}(\overline{M}_{0,n},D_{n_1,n_2,\ldots,n_d})=1$ and the unique up to scalar associated section to $D_{n_1,n_2,\ldots,n_d}$ is in any generator set of $\operatorname{Cox}(\overline{M}_{0,n})$.
\begin{figure}[ht]     
\begin{center}
\begin{minipage}{0.45\linewidth}
\begin{center}
\includegraphics[height=3cm]{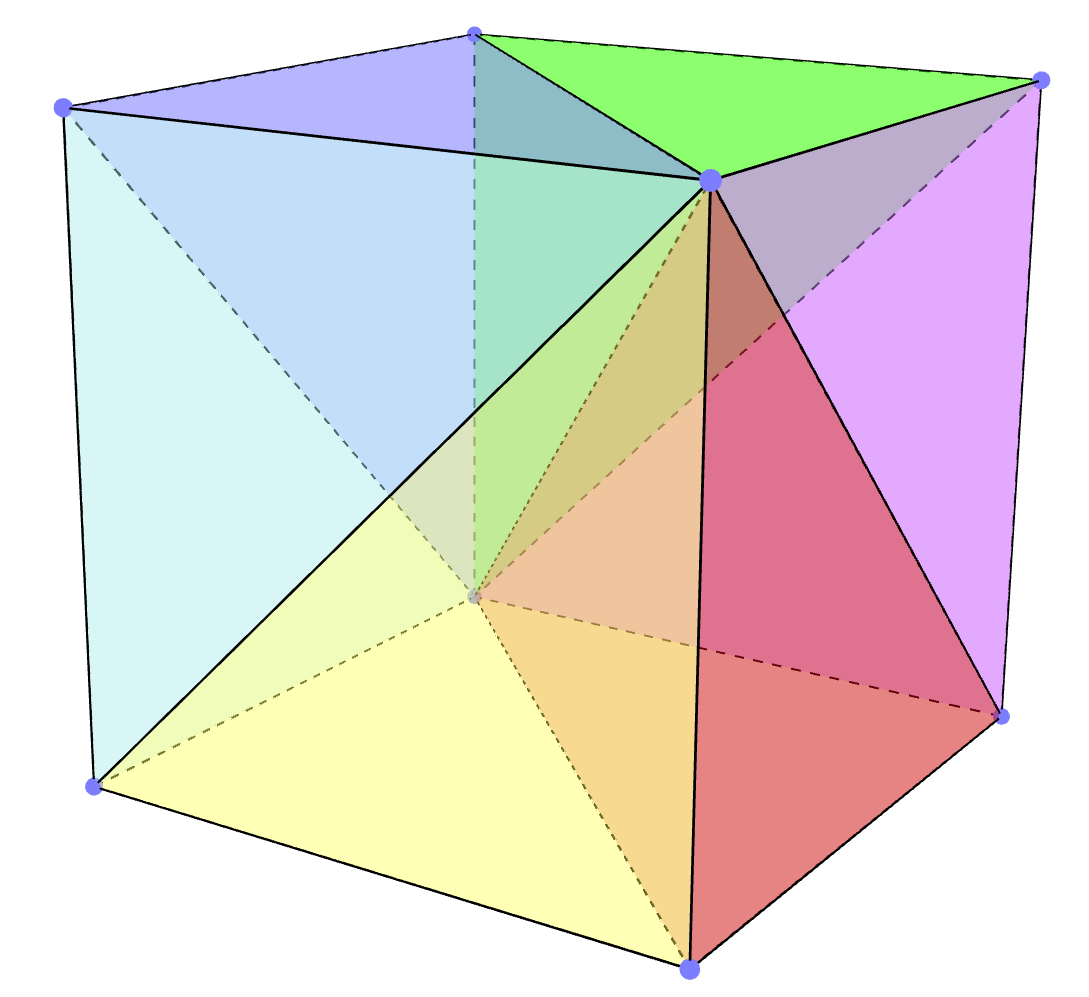} 
\end{center}
\caption{Triangulation $Q_3$.}    \label{fig-Cube-Perspective}
\end{minipage}
\begin{minipage}{0.45\linewidth}
\begin{center}
\includegraphics[height=3cm]{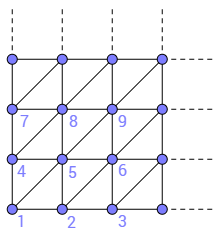}
\end{center}
\caption{Complex $\Delta_{n_1,n_2}$.}   \label{fig-Torus}
\end{minipage}
\end{center}
\end{figure}
To see this, first we consider the weighted $d$-complex $Q_d$ on the vertex set  $\{0,1\}^d \subseteq \mathbb{R}^d$ of the unit $d$-cube, 
with the $d!$ simplices $\bar{\sigma} = \{(x_1,x_2, \ldots ,x_d) \in \{0,1\}^{d}\:| \: x_{\sigma(1)}\leq x_{\sigma(2)}\leq \ldots \leq x_{\sigma(d)} \}$ weighted by $w_{\bar{\sigma}} = \operatorname{sgn}(\sigma)$, for each permutation $\sigma \in S_d$ (see Figure~\ref{fig-Cube-Perspective}).
We get an induced weighted $d$-complex on the lattice points of the hyperrectangle $H_{n_1,n_2,\ldots,n_d}=[0,n_1] \times [0,n_2] \times \ldots \times [0,n_d]$ where each unit $d$-cube with lattice vertices contained in $H_{n_1,n_2,\ldots,n_d}$ is triangulated by a translation of $Q_d$, 
and where a $d$-simplex in the triangulation of the lattice $d$-cube $[a_{1},a_{1}+1] \times [a_{2},a_{2}+1] \times \ldots \times [a_{d},a_{d}+1]$ corresponding to the permutation $\sigma \in S_{d}$ is assigned the weight $(-1)^{(a_1+a_2+\cdots+a_d)d}\operatorname{sgn}(\sigma)$.
By identifying pairs of opposite faces of $H_{n_1,n_2,\ldots,n_d}$ by translations in the direction of the coordinate axes, we get a nonsingular $d$-complex $\Delta_{n_1,n_2,\ldots,n_d}$ on $n_1 n_2 \cdots n_d$ vertices that is a triangulation of the $d$-torus (see Figure~\ref{fig-Torus}).
Each facet on $\Delta_{n_1,n_2,\ldots,n_d}$ is contained in exactly two simplices and with weights 1 and -1, so $\Delta_{n_1,n_2,\ldots,n_d}$ is nondegenerately balanced by Theorem~\ref{overall}. 
A weighting of $\Delta_{n_1,n_2,\ldots,n_d}$ with one weight equal to zero has to be identically zero, so it is minimal.
Each vertex of $\Delta_{n_1,n_2,\ldots,n_d}$ shares a simplex with exactly $2^{d+1}-2$ other vertices,
and since $\Delta_{n_1,n_2,\ldots,n_d}$ has $n_1n_2\cdots n_d$ vertices, by Proposition \ref{nonproduct} it cannot be a product, except possibly in the two cases $d=2$, $n_1=3$, $n_2=3$ and $d=2$, $n_1=3$, $n_2=4$, but it is not a product in those cases (because otherwise, by Example~\ref{link.product} we would be able to express the link of some vertex as the product of two $0$-complexes, which is impossible as these links are hexagons).
It is straightforward to see that the associated divisor class $D_{\Delta_{n_1,n_2,\ldots,n_d}}$ is given by the formula in (\ref{class.d.tori}) and then the claims follow from Theorem~\ref{thm.simplicial.approach} since $\Delta_{n_1,n_2,\ldots,n_d}$ is a nonproduct, nonsingular, minimal, balanced $d$-complex over a field.
\end{example}

\begin{remark}
The case $d =2$, $n_1 = n_2 = 3$ of the construction in Example~\ref{example.appendix} appeared previously in \cite[4.2.6]{DGJ}.
It would be interesting to study degenerate versions of this construction.
For example, the cycle of $m$ tetrahedra in \cite[4.2.2]{DGJ} is isomorphic to the degenerate case of this construction with $d=2$, $n_1=2$, $n_2 = m$.
The even length cycles in \cite[Proposition 4.6]{DGJ} correspond to a degenerate case of this construction with $d=1$, but note that for length $n=4$ this complex is a product. 
\end{remark}


\section{An effective divisor from a singular complex irreducible in $M(\overline{M}_{0,7})$}            \label{section.theorem}

Let $\mathcal{A}$ be the singular simplicial 2-complex on $\{1,2,3,4,5,6\}$ given by
\[
\mathcal{A}=\{
\{1,1,2\},\!
\{1,1,3\},\!
\{1,2,4\},\!
\{1,2,5\},\!
\{1,3,5\},\!
\{1,3,6\},\!
\{1,4,5\},\!
\{1,5,6\},\!
\{2,3,4\},\!
\{2,3,6\},\!
\{2,5,6\},\!
\{3,4,5\}
\}.
\]
The complex $\mathcal{A}$ can be visualized by identifying the vertices with the same labels and the edges between identified vertices in Figure~\ref{fig-complex-A}.
\begin{figure}[ht]     %
\begin{center}
\includegraphics[height=3.5cm]{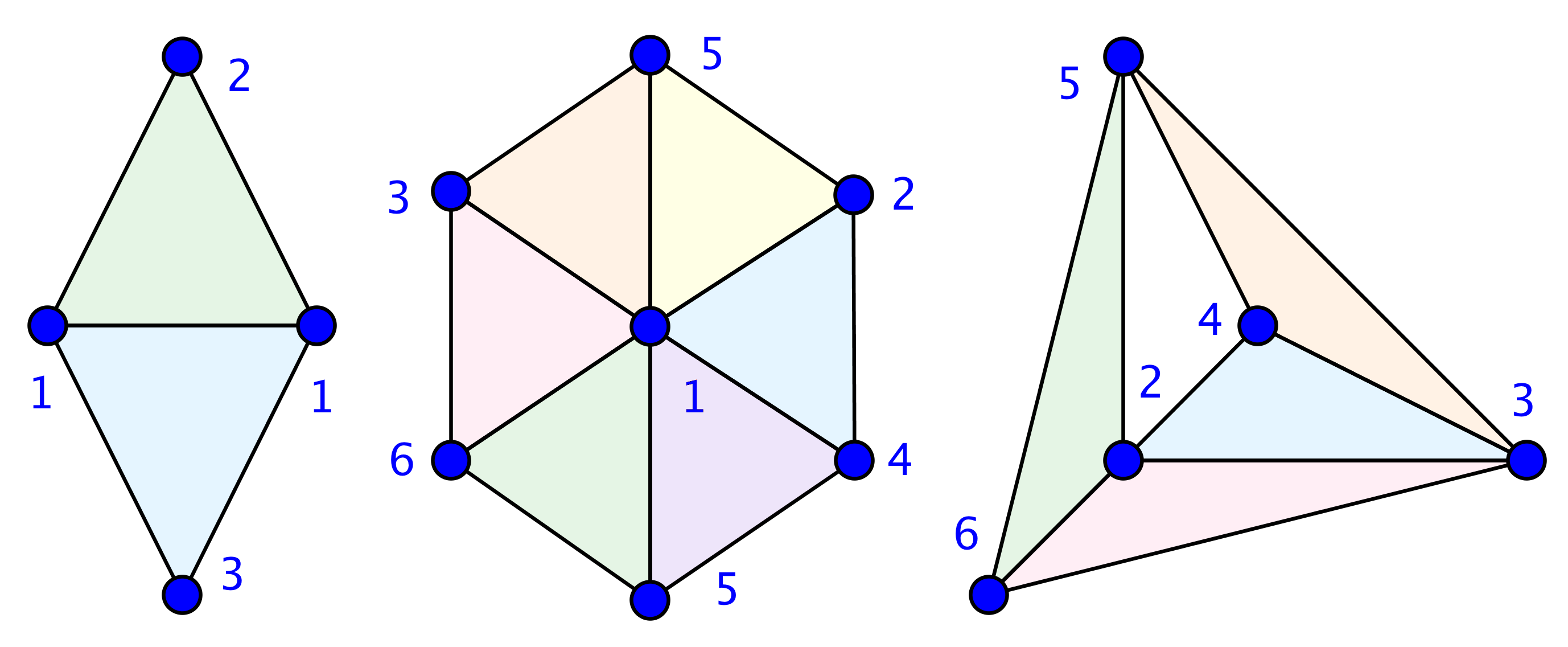} 
\caption{Simplicial 2-complex $\mathcal{A}$ on $\{1,2,3,4,5,6\}$ before vertex and edge identifications.}    \label{fig-complex-A}
\end{center}
\end{figure}

The associated divisor divisor class in $\operatorname{Pic}(\overline{M}_{0,7})$ is
\begin{align*}
D_{\mathcal{A}} = {} & 3H -E_{1}-2E_{2}-2E_{3}-2E_{4}-2E_{5}-2E_{6} 
-E_{1 4}-E_{1 5}-E_{1 6}-E_{2 3}-E_{2 4}-E_{2 5}-E_{2 6}    \\
&   -E_{3 4}-E_{3 5}-E_{3 6}-E_{4 5} -2E_{4 6}-E_{5 6} 
-E_{1 4 6}-E_{2 3 5}-E_{2 4 5}-E_{2 4 6}-E_{3 4 6}-E_{3 5 6}-E_{4 5 6}.
\end{align*}

\begin{lemma}         \label{lemma.effective}
The complex $\mathcal{A}$ is minimal over any characteristic zero field $K$ and its associated divisor $D_{\mathcal{A}}$ represents an effective class in $\overline{M}_{0,7}$.
\end{lemma}
\begin{proof}
Suppose that $w:\mathcal{A} \rightarrow K$ is a balancing of $\mathcal{A}$. The balancing conditions on the facets imply that
\begin{align*}
&w(\{1,2,5\})=-w(\{2,5,6\})=w(\{1,5,6\})=-w(\{1,3,6\})=w(\{2,3,6\}),   \\
&w(\{1,3,5\})=-w(\{3,4,5\})=w(\{1,4,5\})=-w(\{1,2,4\})=w(\{2,3,4\}),   \\
&w(\{1,2,5\})+w(\{1,3,5\})+w(\{1,4,5\})+w(\{1,5,6\})=0,   \\
&w(\{1,1,2\})=-w(\{1,1,3\})   \textnormal{ \ and \ } 2 \cdot w(\{1,1,2\}) + w(\{1,2,4\}) + w(\{1,2,5\})=0.
\end{align*}
From these equations, we deduce that up to scalar $\mathcal{A}$ admits a unique balancing given by
\begin{align*}
&
w(\{1,1,2\})=1,
w(\{1,1,3\})=-1,
w(\{1,2,4\})=-1,
w(\{1,2,5\})=-1,
w(\{1,3,5\})=1,
w(\{1,3,6\})=1, 
   \\
&   
w(\{1,4,5\})=1,  
w(\{1,5,6\})=-1,
w(\{2,3,4\})=1,
w(\{2,3,6\})=-1,
w(\{2,5,6\})=1,
w(\{3,4,5\})=-1.
\end{align*}
Therefore $\mathcal{A}$ is minimal as this unique, up to scalar, balancing is not degenerate. The associated class $D_{\mathcal{A}}$ is effective in $\overline{M}_{0,7}$ by  Theorem~\ref{thm.simplicial.approach}.  
\end{proof}




We now show that $D_{\mathcal{A}}$ is an irreducible element in the of the monoid of effective divisors of $\overline{M}_{0,7}$.

\begin{theorem}[See Remark~\ref{remark.credit}]                         \label{theorem.hypertree}
Over any characteristic zero field $K$ the effective divisor class
\begin{align}   \label{divisor.class}
D_{\mathcal{A}} = {} & 3H -E_{1}-2E_{2}-2E_{3}-2E_{4}-2E_{5}-2E_{6} 
-E_{1 4}-E_{1 5}-E_{1 6}-E_{2 3}-E_{2 4}-E_{2 5}-E_{2 6}   \nonumber    \\
&   -E_{3 4}-E_{3 5}-E_{3 6}-E_{4 5} -2E_{4 6}-E_{5 6} 
-E_{1 4 6}-E_{2 3 5}-E_{2 4 5}-E_{2 4 6}-E_{3 4 6}-E_{3 5 6}-E_{4 5 6}
\end{align}   
in $\operatorname{Pic}(\overline{M}_{0,7})$ is irreducible in the of the monoid of effective divisors of $\overline{M}_{0,7}$.
Moreover,  $h^{0}(\overline{M}_{0,7},D_{\mathcal{A}})=1$ and its unique up to scalar nonzero section is part of any generating set of $\operatorname{Cox}(\overline{M}_{0,7})$.
\end{theorem}
\begin{proof}
Let us write $D_{\mathcal{A}} = D_1+D_2+\ldots+D_r$ where each $D_{i}$ is an effective class irreducible in $M(\overline{M}_{0,7})$.

\textbf{Step 1.} \emph{We show that for any $I \subseteq \{1,2,3,4,5,6\}$ with $1\leq |I| \leq 3$ the divisor $D_{\mathcal{A}} - E_{I}$ does not represent an effective class in $\overline{M}_{0,7}$.}
The pseudoeffective cone of the projective variety $\overline{M}_{0,7}$ is pointed, therefore we can choose an effective integral combination of exceptional divisors $E$ such that $D = D_{\mathcal{A}} - E$ is effective, but $D - E_{I}$ is not effective for any $I \subseteq \{1,2,3,4,5,6\}$ with $1\leq |I| \leq 3$.
By Theorem~\ref{thm.simplicial.approach} there exists a $2$-complex $\Delta$ on $\{1,2,3,4,5,6\}$ such that $D_{\Delta}=D$, where $D_{\Delta}$ is the divisor associated to $\Delta$ according to the formula in (\ref{DGJrecipe2}).
Let us write $D_{\mathcal{A}}= 3H - \sum_{I}a_{I}E_{I}$ and $D_{\Delta}= 3H - \sum_{I}b_{I}E_{I}$. Hence, for fixed $I$ and $\sigma \in \Delta$ we have
\[
a_I \leq  b_I   =  3 - \max_{\tau \in \Delta} \Bigg\{ \sum_{i\in I} m(i \in \tau)\Bigg\}    \leq    3 - \sum_{i\in I} m(i \in \sigma).
  \]
Therefore, any $\sigma \in \Delta$ satisfies
$
\sum_{i\in I} m(i \in \sigma)   \leq   3 - a_I,
$
for all $I$.
Hence the nonsingular simplices $\{1,4,6\}$, $\{2,3,5\}$, $\{2,4,5\}$, $\{2,4,6\}$, $\{3,4,6\}$, $\{3,5,6\}$ and $\{4,5,6\}$ cannot be in $\Delta$. Similarly, the only singular simplices that could be in $\Delta$ are $\{1,1,2\}$ and $\{1,1,3\}$.
Then $\Delta \subseteq \tilde{\Delta} :=\mathcal{A} \cup  \{\{1,2,3\}, \{1,2,6\}, \{1,3,4\}\}$.
By Theorem~\ref{thm.simplicial.approach} there exists a possibly degenerate balancing $w: \tilde{\Delta} \rightarrow K$, which becomes a nondegenerate balancing when restricted to $\Delta$. 
If we add the balancing conditions on $\tilde{\Delta}$ for $\{1,6\}$ and $\{2,6\}$, and subtract those for $\{3,6\}$ and $\{5,6\}$ we get
\begin{align*}
0= ( w(\{1,2,6\}) + &  w(\{1,3,6\}) + w(\{1,5,6\})) + (w(\{1,2,6\}) + w(\{2,3,6\}) + w(\{2,5,6\}))       \\
& - ( w (\{1,3,6\}) + w(\{2,3,6\})) - (w(\{1,5,6\}) + w(\{2,5,6\}))   =  2 \cdot w(\{1,2,6\}).  
\end{align*}
Hence $w(\{1,2,6\})=0$ and $ \{1,2,6\}   \notin \Delta$. Similarly,  if we add the balancing conditions on $\tilde{\Delta}$ for $\{1,4\}$ and $\{3,4\}$, and subtract those for $\{2,4\}$ and $\{4,5\}$ we get
\begin{align*}
0= (w(\{1,2,4\}) + &  w(\{1,3,4\}) + w(\{1,4,5\})) + (w(\{1,3,4\}) + w(\{2,3,4\}) + w(\{3,4,5\}))       \\
& - (w(\{1,2,4\}) + w(\{2,3,4\})) - (w(\{1,4,5\}) + w(\{3,4,5\}))   =  2 \cdot w(\{1,3,4\}).  
\end{align*}
Hence $w(\{1,3,4\})=0$ and $ \{1,3,4\}   \notin \Delta$. 
If we add the balancing conditions on $\tilde{\Delta}$ for $\{1,2\}$, $\{1,3 \}$, $\{2, 6 \}$ and $\{3,4\}$, 
subtract those for $\{2, 4\}$, $\{2, 5 \}$, $\{3, 5 \}$ and $\{3,6\}$, 
and subtract twice the one for $\{1,1\}$ we get
\begin{align*}
0 = ( 2 \cdot &  w(\{1,1,2\}) +  w(\{1,2,3\}) + w(\{1,2,4\}) +  w(\{1,2,5\}) +  w(\{1,2,6\}) )      \\
& + ( 2 \cdot  w(\{1,1,3\}) +  w(\{1,2,3\}) + w(\{1,3,4\}) +  w(\{1,3,5\}) +  w(\{1,3,6\}) )      \\
& + (w(\{1,2,6\}) + w(\{2,3,6\}) + w(\{2,5,6\}))    
 + ( w(\{1,3,4\}) + w(\{2,3,4\}) + w(\{3,4,5\}))      \\
&- ( w(\{1,2,4\}) +  w(\{2,3,4\})) 
- ( w(\{1,2,5\}) +  w(\{2,5,6\}) )
- (w(\{1,3,5\}) + w(\{3,4,5\}) )       \\
& - (w(\{1,3,6\}) + w(\{2,3,6\}) )
- 2 \cdot  ( w(\{1,1,2\}) +  w(\{1,1,3\})). 
\end{align*}
Hence $2 \cdot w(\{1,2,3\}) + 2 \cdot  w(\{1,2,6\}) + 2 \cdot  w(\{1,3,4\}) = 0$, and then $w(\{1,2,3\})= 0 $ and  $ \{1,2,3\}   \notin \Delta$. 
Therefore, $\Delta \subseteq \mathcal{A}$. The class $3H - \sum_{I} 3E_I$ is not effective in $\overline{M}_{0,7}$, then $\Delta$ is not empty. 
Since $\mathcal{A}$ is minimal by Lemma~\ref{lemma.effective}, we deduce that $\mathcal{A}=\Delta$ and $D_{\mathcal{A}}=D_{\Delta}=D$,
and the desired conclusion follows.

\textbf{Step 2.} \emph{We show $D_{\mathcal{A}}$ is irreducible in $M(\overline{M}_{0,7})$.} 
By contradiction, let us assume that $D_{\mathcal{A}}$ is reducible. 
Let us denote by $B_{i_1i_2}$, $H_{i_1i_2i_3i_4i_5i_6}$, $T_{(i_1)(i_2i_3)(i_4i_5)}$ and $P_{(i_1)(i_2i_3)(i_4)(i_5i_6)}$ the divisors corresponding to the corresponding complexes from Examples~\ref{example.degree.one} and \ref{example.degree.two}.
By Step 1, all divisors $D_{i}$ in the expression $D_{\mathcal{A}} = D_1+D_2+\ldots+D_r$ from above are among the degree one and two divisors described in Examples~\ref{example.degree.one} and \ref{example.degree.two}.
Therefore, one of the following three decompositions must hold
\begin{align*}
D_{\mathcal{A}}&=B_{i_1i_2}+B_{j_1j_2}+B_{k_1k_2},    \qquad \qquad
 &D_{\mathcal{A}}&=B_{i_1i_2}+H_{j_1j_2j_3j_4j_5j_6},   \\
D_{\mathcal{A}}&=B_{i_1i_2}+T_{(j_1)(j_2j_3)(j_4j_5)},    \qquad \qquad
&D_{\mathcal{A}}&=B_{i_1i_2}+P_{(j_1)(j_2j_3)(j_4)(j_5j_6)},
\end{align*}
where the subindices of each divisor are distinct elements from the set $\{1,2,3,4,5,6\}$.
However, these equations are impossible because the sum of the coefficients of the divisors $E_{I}$ with $|I|=3$ on $D_{\mathcal{A}}$ is -7, on a divisor of the form $B_{i_1i_2}$ is $-4$, 
on a divisor of the form $H_{j_1j_2j_3j_4j_5j_6}$ is $-2$, 
on a divisor of the form $T_{(j_1)(j_2j_3)(j_4j_5)}$ is $-4$,
and on a divisor of the form $P_{(j_1)(j_2j_3)(j_4)(j_5j_6)}$ is $0$.
Therefore $D_{\mathcal{A}}$ is irreducible in $M(\overline{M}_{0,7})$ as we wanted to show.

\textbf{Step 3.} \emph{We show $h^{0}(\overline{M}_{0,7},D_{\mathcal{A}})=1$ and the claim on the section.} We proceed as in \cite[Proposition 2.21]{DGJ}. A nonzero section $f \in H^{0}(\overline{M}_{0,7},D_{\mathcal{A}}) \subseteq \operatorname{Cox}()\overline{M}_{0,7}$ cannot be divisible by any exceptional divisor section by Step 1. 
By Theorem~\ref{thm.simplicial.approach}, $f$ corresponds to a nondegenerately balanced complex $\Delta$ such that $D_{\mathcal{A}}=D_{\Delta}$.
By repeating the argument in Step 1 with $E=0$, we conclude that $\mathcal{A}=\Delta$. 
Hence, the nonzero elements in $H^{0}(\overline{M}_{0,7},D_{\mathcal{A}})$ correspond to nondegenerate balancings of $\mathcal{A}$. By Lemma~\ref{lemma.effective}, we get that $h^{0}(\overline{M}_{0,7},D_{\mathcal{A}})=1$, and from the irreducibility of $D_{\mathcal{A}}$ in $M(\overline{M}_{0,7})$, its unique up to scalar nonzero section is part of any generating set of $\operatorname{Cox}(\overline{M}_{0,7})$.
\end{proof}

\begin{remark}     \label{remark.credit}
The divisor $D_{\mathcal{A}}$ studied in this section is actually the \emph{hypertree divisor} corresponding to the unique \emph{hypertree graph} on seven vertices \raisebox{-0.5mm}{{\includegraphics[angle=0,origin=c, height=\myheight]{Hypertree7}}} (see \cite{Hypertrees}).
Hence, the conclusions of Theorem~\ref{theorem.hypertree} easily follow from \cite[Theorem 1.5]{Hypertrees}.
There is a Kapranov model of $\overline{M}_{0,7}$ where $D_{\mathcal{A}}$ becomes the divisor class associated to a hexagon 
\raisebox{-\mydepth}{{\includegraphics[height=\myheight]{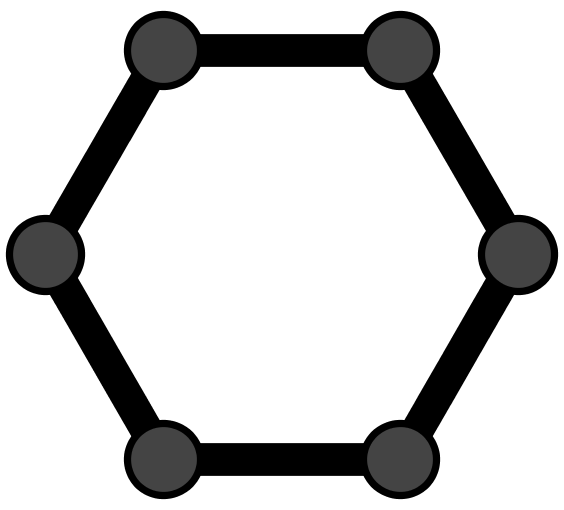}}} viewed as a nonsingular $1$-complex.
%
%
Thus, the conclusions of Theorem~\ref{theorem.hypertree} are directly given by the nonsingular case of Theorem~\ref{thm.simplicial.approach}.
We believe that arguments that allow to prove the irreducibility in the effective monoid of divisors arising from balanced singular complexes are 
important toward a solution to Question 1.
We expect that the ideas in this section can be used to prove irreducibility in $M(\overline{M}_{0,n})$ of effective divisors arising from balanced singular complexes,
in some cases where \cite[Theorem 1.5]{Hypertrees} and Theorem~\ref{thm.simplicial.approach} do not apply.
\end{remark}


\section{The space of balancings}   \label{section.balancings}

We define the \emph{complete $d$-complex on $n$ vertices} $\Delta_{n,d}$ as the $d$-complex whose simplices are all $d$-simplices on the set $\{1,2,\ldots, n \}$. 
Similarly, we define the \emph{complete nonsingular $d$-complex on $n$ vertices} $\Delta^{ns}_{n,d}$ as the $d$-complex whose simplices are all nonsingular $d$-simplices on the set $\{1,2,\ldots, n \}$. Note that all $d$-complexes on $n$ or fewer vertices are subcomplexes of $\Delta_{n,d}$ and all nonsingular $d$-complexes on $n$ or fewer vertices are subcomplexes of $\Delta^{ns}_{n,d}$.
A balancing on a subcomplex of $\Delta_{n,d}$ (resp. $\Delta^{ns}_{n,d}$) can be extended by zero to get a balancing of $\Delta_{n,d}$ (resp. $\Delta^{ns}_{n,d}$).
In fact, we see that the set of nonzero balancings of $\Delta_{n,d}$ (resp. $\Delta^{ns}_{n,d}$) correspond to the set of nondegenerate balancings on subcomplexes
of $\Delta_{n,d}$ (resp. $\Delta^{ns}_{n,d}$).
In this section we study the vector spaces of balancings of the complexes $\Delta_{n,d}$ and $\Delta^{ns}_{n,d}$ over a field $K$. In particular, we compute their dimensions and give explicit bases. 
These bases can be used to construct families of nondegenerate balanced complexes ad hoc and by exhaustive searches.

\begin{notation}
In this section, we will occasionally represent a given weighted $d$-complex $(\Delta,w)$ on the set $\{1,2,\ldots, n \}$ as the formal linear combination of the simplices in $\Delta_{n,d}$ given by  
$\sum_{\sigma \in \Delta} w(\sigma) \, \sigma$.
Given any weighted $d$-complex $\Delta_1=(\Delta,w)$ on $\{1,2,\ldots, n \}$ and a nonnegative integer $d'$, we associate to $\Delta_1$ the linear functional on the space of weighted $d'$-complexes on $\{1,2,\ldots, n \}$ which on any such $d'$-complex $\Delta_2=(\Delta',w')$ takes the value 
\[
\Delta_1(\Delta_2):=\sum_{\sigma \in \Delta} \sum_{\sigma' \in \Delta'}  w(\sigma)   \,   w'(\sigma')   \,  m(\sigma \subseteq \sigma').
\]
 \end{notation}

\begin{definition}[The multiplicity matrix]
Let $\Delta =  \{\sigma_1, \sigma_2, \ldots, \sigma_r \}$ be a nonempty $d$-complex on $\{1,2,\ldots, n \}$ and let 
$S =  \{\tau_1, \tau_2, \ldots, \tau_l \}$ be a nonempty collection of multisets on $\{1,2,\ldots, n \}$, both ordered in lexicographic order unless specified otherwise. 
The \emph{multiplicity matrix} $M(\Delta,S)$ is the $l \times r$ matrix 
\[
  M(\Delta,S) 
=
\begin{pmatrix}
    m(\tau_1 \subseteq \sigma_1) & m(\tau_1 \subseteq \sigma_2) & \dots  & m(\tau_1 \subseteq \sigma_r) \\
    m(\tau_2 \subseteq \sigma_1) & m(\tau_2 \subseteq \sigma_2) & \dots  & m(\tau_2 \subseteq \sigma_r) \\
    \vdots & \vdots & \ddots & \vdots \\
    m(\tau_l \subseteq \sigma_1) & m(\tau_l \subseteq \sigma_2) & \dots  & m(\tau_l \subseteq \sigma_r)
\end{pmatrix}
\]
whose entry in the $i$-th row and $j$-th column is $ M(\Delta,S)_{ij} = m(\tau_i \subseteq \sigma_j) $.
If we take $S$ to be the collection of all multisets of cardinality at most $d$ on $\{1,2,\ldots, n \}$, then by definition, for any ring $R$ the kernel of the matrix $M(\Delta,S)$ in $R^r$ is the $R$-submodule of $R$-balancings of $\Delta$.
By Theorem~\ref{overall}, over any characteristic zero domain $R$, the kernel of the matrix $M(\Delta,\Delta_{n,d-1})$ in $R^r$ is precisely the $R$-submodule of $R$-balancings of $\Delta$.
\end{definition}


\subsection{A basis for the space of balancings}

\begin{proposition}
The vector space of balancings over a characteristic zero field $K$ on the complete $d$-complex on $n$ vertices has dimension $\binom{n+d-1}{d+1}$.
\end{proposition}

\begin{proof}
For $n=1$ or $d=0$ the claim can be easily verified, so we will assume $n \geq 2$ and $d \geq 1$.
We now show that the rows of the multiplicity matrix $M(\Delta_{n,d},\Delta_{n,d-1})$ are linearly independent.
If these rows are not linearly independent, there exists a nonzero vector $(a_1,a_2,\ldots,a_l) \in K^l$ such that for each $\sigma \in \Delta_{n,d}$ we have
\begin{equation}   \label{eqn.combination}
a_1 m(S_1 \subseteq \sigma)+a_2m(S_2 \subseteq \sigma)+\cdots+a_lm(S_l \subseteq \sigma)=0.
\end{equation}
where $S_1,S_2,\ldots,S_l$ are all the multisets of size $d$ of $\{1,2,\ldots, n \}$ ordered lexicographically.
Let $1 \leq i_0 \leq l$ be the smallest positive integer such that $a_{i_0}$ is not zero.
We have that no facet of the $d$-simplex $\sigma= S_{i_0} \uplus \{ 1 \}$ comes after $S_{i_0}$ in lexicographic order.
Then from (\ref{eqn.combination}), it follows that $a_{i_0} m(S_{i_0} \subseteq S_{i_0} \uplus \{ 1 \}) =0$. Hence $a_{i_0}=0$, which is a contradiction.
Since the rows of $M(\Delta_{n,d},\Delta_{n,d-1})$ are linearly independent, the claim now follows by computing the dimension of its kernel as 
 $|\Delta_{n,d}|-|\Delta_{n,d-1}| = \binom{n+d}{d+1} - \binom{n+d-1}{d} = \binom{n+d-1}{d+1}$.
\end{proof}

\begin{proposition}  A basis for the vector space of balancings over a characteristic zero field on the complete $d$-complex on $n$ vertices is given by all balanced complexes of the form  $\prod_{j=2}^{n} \{\{1\},\{j\}\}^{k_j}$, where each $k_j \in \mathbb{Z}_{\geq 0}$, $\sum {k_j}=d+1$, with the product balancing induced by balancing $\{\{1\},\{j\}\}$ with weights $w(\{1\})=1$, $w(\{j\})=-1$ .    
\end{proposition}

\begin{proof}
The proposed basis indeed consists of balanced complexes by Example~\ref{example.balancing.product} and it has the right cardinality $\binom{n+d-1}{d+1}$, so it suffices to show its linear independence.
Given a nonzero linear combination of these weighted complexes, considered as elements of the vector space of balancings, 
the last $d$-simplex in lexicographic order having a nonzero coefficient occurs in exactly one of the proposed elements. Hence, the linear combination cannot be equal to zero and their linear independence follows.    
\end{proof}


\subsection{A basis for the space of nonsingular balancings}

\begin{proposition}\label{nonsingdimprop}
The vector space of balancings over a characteristic zero field $K$ on the complete nonsingular $d$-complex on $n$ vertices $\Delta^{ns}_{n,d}$ has dimension $\operatorname{max}\{\binom{n}{d+1} - \binom{n}{d},0\}$.
\end{proposition}
\begin{proof} 
Note that in the nonsingular case the simplices are sets. 
It is straightforward to verify the statement holds in the case $d=0$, where the dimension is $n-1$, and in the case $n \leq d+1$, where the dimension is zero.
Then, from here on we assume that $d \geq 1$ and $n \geq d+2$.
Let us suppose first that $\binom{n}{d+1} - \binom{n}{d} \leq 0$, i.e. we assume that $2d+1 \geq n$. 
We claim that for each $\sigma_{0} \in \Delta^{ns}_{n,d}$ there exist constants $a_{\sigma_{0},\tau}\in K$ for all $\tau \in \Delta^{ns}_{n,d-1}$, such that
\[
\sum_{\tau \in \Delta^{ns}_{n,d-1}} a_{\sigma_{0},\tau} \,  \tau(\sigma)= 
\begin{cases}
1 &\textnormal{if $\sigma =\sigma_{0} $}, \\
0 &\textnormal{if $\sigma \in \Delta^{ns}_{n,d} \smallsetminus \{\sigma_{0} \} $}.
\end{cases}
\]
To prove the claim fix $\sigma_{0} \in \Delta^{ns}_{n,d}$, and for each $1 \leq j \leq n-d$ we define
\[
A_j=\{\sigma \in \Delta^{ns}_{n,d} : | \sigma_{0} \smallsetminus \sigma| = j-1 \}
\]
and for each $1 \leq i \leq n-d$ we define 
\[
B_i=\{\tau \in \Delta^{ns}_{n,d-1} : | \sigma_{0} \smallsetminus \tau| = i \}
\]
and we set $f_{i}=\sum_{\tau \in B_i} \tau$.
Note that $2d-n+2 \leq | \sigma \cap \sigma_{0}| \leq d+1$ and $2d-n+1 \leq | \tau \cap \sigma_{0}| \leq d$ for any $\sigma \in \Delta^{ns}_{n,d}$ and $\tau \in \Delta^{ns}_{n,d-1}$.
Therefore, the sets $A_j$ and $B_i$ for $1 \leq i,j \leq n-d$ give partitions $\Delta^{ns}_{n,d}=\sqcup A_j$ and $\Delta^{ns}_{n,d-1}=\sqcup B_i$, according to the size of the intersection that each simplex has with $\sigma_{0}$.
By symmetry each $f_i$ takes a constant value $c_{ij}$ on all the elements of any particular set $A_{j}$. The definition of the $A_{j}$ and $B_i$ implies that the square matrix $C=(c_{ij})_{1\leq i,j\leq n-d}$ is upper triangular with nonzero entries on the diagonal. 
Thus, there exists a linear combination $f$ of the $f_{i}$ such that $f(\sigma) = 1$ if $\sigma \in A_1$, and $f(\sigma)=0$ if $\sigma \in A_j$, for any $j>1$. Since $A_1=\{\sigma_{0}\}$, the claim is proved.
The claim implies that in this case the only balancing on the complete nonsingular $d$-complex is identically zero as desired. 
Now let us suppose that $\binom{n}{d+1} - \binom{n}{d} \geq 0$, i.e. we assume that $n \geq 2d +1$.             
Note that by taking complements in $\{1,2,\ldots,n \}$, any simplices $\sigma \in \Delta^{ns}_{n,d}$ and $\tau \in \Delta^{ns}_{n,d-1}$ give us simplices $\sigma^{c} \in \Delta^{ns}_{n,n-d-2}$ and $\tau^{c} \in \Delta^{ns}_{n,n-d-1}$. 
Also, note that $d':=n-d+1$ satisfies the inequalities on the degree assumed for the claim above, namely, $d' \geq 1$ and $2d'+1 \geq n \geq d'+2$.
Using that $\sigma^{c}(\tau^{c})=\tau(\sigma)$, it follows from the claim applied to the simplices in $\Delta^{ns}_{n,n-d-1}$ that for each $\tau_0 \in \Delta^{ns}_{n,d-1}$ there exist constants $b_{\sigma,\tau_{0}}\in K$ for all $\sigma \in \Delta^{ns}_{n,d}$, such that
\[
\sum_{\sigma \in \Delta^{ns}_{n,d}}b_{\sigma,\tau_0}\tau(\sigma)= 
\begin{cases}
1 &\textnormal{if $\tau =\tau_{0} $}, \\
0 &\textnormal{if $\tau \in \Delta^{ns}_{n,d-1} \smallsetminus \{\tau_{0} \} $}.
\end{cases}
\]
This proves the linear independence of the $\tau \in \Delta^{ns}_{n,d-1}$ as linear functionals on the vector space of formal linear combinations of the elements of $\Delta^{ns}_{n,d}$. Therefore the dimension of the vector space of balancings on the complete nonsingular $d$-complex is $|\Delta^{ns}_{n,d}|-|\Delta^{ns}_{n,d-1}|=\binom{n}{d+1} - \binom{n}{d}$, as desired.
\end{proof}

\begin{corollary}  \label{corollary.no.nonsingular}
Over a characteristic zero field there are no nonempty nonsingular balanceable $d$-complexes on $n$ vertices when $d+1 > \frac{n}{2}$.
\end{corollary}


\begin{proposition}
A basis for the vector space of balancings over a characteristic zero field on the complete nonsingular $d$-complex on $n$ vertices is given by all balanced complexes of the form 
$\prod_{j=1}^{d+1} \{\{a_j\},\{b_j \}\}$, for each possible choice of integers $1 \leq a_1 < a_2 < \cdots < a_{d+1} \leq n$ with $a_j \geq 2j$ for each $j$ and where $1 \leq b_1 < b_2 < \cdots < b_{d+1} \leq n$ are the $d+1$ smallest elements of $\{1,2,\ldots,n\} \smallsetminus \{a_1,a_2,\ldots,a_{d+1}\}$, 
with the product balancing induced by balancing each $\{\{a_j\},\{b_j \} \}$ with weights $w(\{a_j\})=1$, $w(\{b_j\})=-1$.
\end{proposition}

\begin{proof}
For any integers $n \geq 1$ and $d \geq 0$, let $B_{n,d}$ be the set of balanced complexes in the statement.
We prove a stronger statement, namely, 
that for any $n \geq 1$ and $d \geq 0$, $B_{n,d}$ is a basis as desired and additionally for each element of $\Delta \in B_{n,d}$ there exists $\tau_{\Delta}$, a linear combination of $d$-simplices on $\{1,2,\ldots, n\}$, such that $\tau_{\Delta}(\Delta)\neq 0$ and $\tau_{\Delta}(\Delta') = 0$ for all $\Delta'\in B_{n,d} \smallsetminus \{ \Delta \}$.
This holds whenever $n < 2d+2$ because in this case the vector space is trivial and $B_{n,d}$ is empty.
This also holds if $d=0$ and $n \geq 2$, because in this case the space has dimension $n-1$,    
$B_{n,0}$ consists of the $n-1$ linearly independent elements $\Delta_{j}=\{\{ j \},\{ 1 \}\}= 1 \cdot \{ j \}  - 1 \cdot \{ 1 \}$ for every integer $2 \leq j \leq n$,    
and we can take $\tau_{\Delta_j}=\{ j \}= 1 \cdot \{ j \}$ for each integer $2 \leq j \leq n$. 
We now deal with the remaining cases $d \geq 1$ and $n \geq 2d+2$, by induction on $d+n$ for $n$ and $d$ in this range.
In the base case $d=1$ and $n=4$, the claim holds as 
$\Delta_1= \{\{2\},\{1 \} \} \cdot   \{\{4\},\{3 \} \}$     
and      
$\Delta_2 =\{\{3\},\{1 \} \}  \cdot    \{\{4\},\{2 \} \}$,      
and we can take the linear combinations
$\tau_{\Delta_1}= \{2, 4 \} = 1 \cdot \{2, 4 \} $     
and     
$\tau_{\Delta_2}= \{3, 4 \} = 1 \cdot \{3, 4 \} $.    
Fix now some $d \geq 1$ and $n \geq 2d+2$, and assume that our induction hypothesis holds for all $d' \geq 1$ and $n' \geq 2d'+2$ such that $n'+d'<n+d$.
We notice that $B_{n-1,d} \subseteq B_{n,d}$. We also observe that the elements in $B_{n-1,d-1}$ are in bijective correspondence with the elements in $B_{n,d} \smallsetminus B_{n-1,d}$, where the bijection sends an element $ \Delta = \prod_{j=1}^{d} \{\{a_j\},\{b_j \}\} \in B_{n-1,d-1}$ to 
$\Delta  \cdot \{\{n \},\{ b \}\}  \in B_{n,d} \smallsetminus B_{n-1,d}$, where $b = \operatorname{min}\left(\{1,2,\ldots,n-1\} \smallsetminus \{a_1,a_2,\ldots,a_{d}, b_1,b_2,\ldots,b_{d}\}\right)$ and $\{\{n\}, \{b\}\}$ has the balancing $w(\{n\})=1$, $w(\{b\})=-1$.
We know that $B_{n-1,d}$ and $B_{n-1,d-1}$ are bases as desired, because each of these either corresponds a case previously established or a case handled by our induction hypothesis.
Using that $n \geq 2d+2$, by Proposition~\ref{nonsingdimprop} we get that
$|B_{n-1,d}| = \binom{n-1}{d+1} - \binom{n-1}{d}$
and 
$|B_{n-1,d-1}| = \binom{n-1}{d} - \binom{n-1}{d-1}$.
Therefore
\[
|B_{n,d}| =|B_{n-1,d}| + |B_{n-1,d-1}| =\left[\binom{n-1}{d+1} - \binom{n-1}{d}\right]+\left[\binom{n-1}{d} - \binom{n-1}{d-1}\right] = \binom{n}{d+1} - \binom{n}{d},
\]
so $B_{n,d}$ has the right cardinality to be a basis of the desired space. 
Given $\Delta \in B_{n,d}$ let us show the existence of the desired $\tau_\Delta$. 
If $\Delta \in B_{n,d} \smallsetminus B_{n-1,d}$, it has the form $\Delta = \Delta' \cdot \{\{n \},\{ b \}\}$ for some $\Delta'$ in $B_{n-1,d-1}$ and then it is straightforward to verify we can take $\tau_{\Delta}=\tau_{\Delta'}    \cdot  \{n\}$.
If $\Delta \in B_{n-1,d} \subseteq B_{n,d}$, by induction there exist 
$\tau'_{\Delta}$, a linear combination of $d$-simplices on $n-1$ vertices, such that $\tau'_{\Delta}(\Delta) \neq 0$ but $\tau'_{\Delta}(\Delta') = 0$ for all $\Delta' \in B_{n-1,d} \smallsetminus \{ \Delta  \} $. It is straightforward to verify we can take
\[
\tau_{\Delta}=\tau'_{\Delta}-\sum_{\Delta' \in B_{n,d} \smallsetminus B_{n-1,d}} \frac{\tau'_{\Delta}(\Delta')}{\tau_{\Delta'}(\Delta')} \cdot \tau_{\Delta'}.
\]
Finally, we notice that the existence of the $\tau_{\Delta}$ implies that $B_{n,d}$ is a linear independent set, which is therefore a basis of the desired space as it has the right cardinality.
\end{proof}


\bibliographystyle{plain}
\bibliography{divs}

\end{document}